\documentclass[amsfonts,12pt]{amsart}
\usepackage{epsfig}
\allowdisplaybreaks[4]
\usepackage{mathtools}
\usepackage[subnum]{cases}
\usepackage{amsfonts}
\usepackage{}
\usepackage{mathrsfs}
\usepackage{amssymb,url}
\usepackage{cite}
\date{}

\allowdisplaybreaks[4]

\newtheorem{theorem}{Theorem}[section]

\newtheorem{proposition}[theorem]{Proposition}
\newtheorem{lemma}[theorem]{Lemma}

\newtheorem{definition}[theorem]{Definition}
\newtheorem{remark}[theorem]{Remark}
\newtheorem{example}[theorem]{Example}

\usepackage[]{hyperref}
\usepackage{tikz}

\usepackage{amsmath,amssymb}
\usepackage{color}
\usepackage{epsfig}
\usepackage{mathtools}
\usepackage[subnum]{cases}
\usepackage{enumerate}
\usepackage{amsmath,amssymb}
\usepackage{color}
\newtheorem{corollary}[theorem]{Corollary}

\allowdisplaybreaks[4]

\usepackage{hyperref}

\newcommand{\Rmnum}[1]{\expandafter\@slowromancap\romannumeral #1@}
\makeatother

\usepackage{epsfig}
\usepackage{mathtools}
\usepackage[subnum]{cases}
\usepackage{enumerate}
\usepackage{amsmath,amssymb}
\usepackage{color}
\global\long\def\epsilon{\varepsilon}%

\global\long\def\phi{\varphi}%

\newcommand{\R}{\mathbb{R}} 

\newcommand{\e}{\varepsilon} 

\newcommand{\tlp}{\times_{p,s}}
\newcommand{\plp}{\oplus_{p,s}}

\newcommand{\al}{\alpha}
\newcommand{\be}{\beta}
\newcommand{\ga}{\gamma}
\newcommand{\cvxb}{\mathcal{K}^n}
\newcommand{\V}{\text{vol}}
\newcommand{\s}{\mathbb{S}}
\newcommand{\supp}{\text{supp}}
\newcommand{\epi}{\text{Epi}}
\newcommand{\sub}{\text{Sub}}

\newcommand{\Q}{\mathbb{Q}} 
\newcommand*\circled[1]{\tikz[baseline=(char.base)]{
		\node[shape=circle,draw,inner sep=1pt] (char) {#1};}}

\begin{document}

	\author{Michael Roysdon and Sudan Xing}
	
	\title{On the framework of $L_{p}$ summations for functions}

	
	\thanks{The first named author was supported in part by the Zuckerman STEM Leadership program}

	\maketitle

	\begin{abstract}
		We develop the framework of $L_p$ operations for functions by introducing  two primary new types  $L_{p,s}$ summations for $p>0$: the $L_{p,s}$ convolution sum and the $L_{p,s}$ Asplund sum for functions. The first type is defined as the linear summations of functions in terms of the $L_p$ coefficients ($C_{p,\lambda,t}$, $D_{p,\lambda,t}$), the so-called the $L_{p,s}$  supremal-convolution when $p\geq1$ and the $L_{p,s}$ inf-sup-convolution when $0<p<1$, respectively. The second type $L_{p,s}$ summation is created by the $L_p$  averages of bases for $s$-concave functions. We show that they are equivalent in the case $s=0$ (log-concave functions) and $p\geq1$.  For the former type $L_{p,s}$ summation, we establish the corresponding $L_p$-Borell-Brascamp-Lieb inequalities  for all $s\in[-\infty,\infty]$ and $p\geq1$. 	Furthermore, in summarizing the conditions for these new types of $L_p$-Borell-Brascamp-Lieb inequalities, we define a series of the $L_{p,s}$ concavity definitions for functions and measures. On the other hand, for the latter type $L_{p,s}$ Asplund summation,  we discover the integral formula for $L_{p,s}$ mixed quermassintegral for functions via tackling
		the variation formula of quermassintegral of functions for $p\geq 1$. 	
		
	\end{abstract}
	
	\tableofcontents
	
	\section{Introduction}
	Following the seminal books and surveys of Gardner \cite{Gar1,Gar2},    Artstein-Avidan, Giannopoulos, and Milman \cite{AGM}, and conventions of Schneider \cite{Sh1},  the Brunn-Minkowski theory of convex bodies and functions will be given firstly as the geometric background.
	
	\subsection{Background for convex bodies}
	We will focus on the $n$-dimensional Euclidean space $\R^n$, together with the origin $``o"$ and the usual Euclidean norm $\|x\|=\sqrt{\langle x,x \rangle }$ where $\langle \cdot,\cdot\rangle$ denotes for the standard inner product for vectors in $\R^n$. 	The unit ball $B_2^n$ whose volume is $\omega_n$  with boundary the unit sphere $S^{n-1}=\partial B_2^n.$
	A subset $K$ of $\R^n$ is said to be a convex body if it is a compact, convex set with non-empty interior (containing the origin $o$), and the set of all convex bodies in $\R^n$ will be denoted as $\cvxb_o$ endowed with the Lebesgue measure (volume) $\V_n(\cdot)$, and  $\cvxb_{(o)}$ denotes those containing the origin in their interiors.

	To each $K \in \cvxb_o$, we associate three correspondingly uniquely determined functions: the convex indicator functions $I_K$, characteristic function $\chi_K$ and the support function $h_K$. The convex indicator function  $I_K$ and characteristic function $\chi_K$ associated to $K \in \cvxb_o$ are  defined, respectively, by 
	\[
	I_K(x) = \begin{cases}
	0, &\text{if } x \in K,\\
	+\infty, &\text{if } x \not\in K.
	\end{cases}\ \text{and}  \	\chi_K(x) = \begin{cases}
	1, &\text{if } x \in K,\\
	0, &\text{if } x \not\in K.
	\end{cases}\  
	\]	
	The support function of $K \in \cvxb_o$, $h_K \colon \s^{n-1} \to \R$ is  defined as 
	$
	h_K(u) = \sup_{y \in K } \langle u,y \rangle. 
	$

	In \cite{Firey}  Firey introduced the following generalization of the Minkowski combination of convex bodies: for $K,L \in \cvxb_{(o)}$, $p\geq1$ and $\alpha, \beta \geq 0$, $\alpha \cdot_p K +_p \beta \cdot_p L$, the \emph{$L_p$ Minkowski sum} is defined as the convex body whose support function is 
	$
	h_{\alpha \cdot_p K +_p \beta \cdot_p L}(u) = \left(\alpha h_K(u)^p+\beta h_L(u)^p \right)^{\frac{1}{p}}
	=\left(h_{\alpha^{1/p}K}(u)^p + h_{\beta^{1/p}L}(u)^p \right)^{\frac{1}{p}}.
	$
	Additionally, Firey established the so-called \emph{$L_p$-Brunn-Minkowski inequality} for convex bodies  when $p \geq 1$:  given $K,L \in \cvxb_{(o)}$ and $\alpha, \beta \geq 0$, 
	$
	\V_n(\alpha \cdot_p K +_p \beta \cdot_p L)^\frac{p}{n} \geq \alpha \V_n(K)^{\frac{p}{n}}+\beta \V_n(L)^\frac{p}{n}.
	$
	The operations  $+_p$ and $\cdot_p$ were generalized in \cite{LYZ} by Lutwak, Yang and Zhang to the setting of non-convex sets (measurable sets) in $\R^n$; i.e., for any $\alpha,\beta \geq 0$ and any measurable subsets $A,B \subset \R^n$, 
	\begin{align}
	\alpha \cdot_p A +_p \beta \cdot_p B \nonumber&= \left\{\alpha^{\frac{1}{p}}(1-\lambda)^{\frac{1}{q}}x + \beta^{\frac{1}{p}}\lambda^{\frac{1}{q}}y \colon x \in A, y \in B, 0 \leq \lambda \leq 1\right\}\\
	\label{lpsummationgeometry}&=\bigcup_{0 \leq \lambda \leq 1} \left(\alpha^{\frac{1}{p}}(1-\lambda)^{\frac{1}{q}}A + \beta^{\frac{1}{p}}\lambda^{\frac{1}{q}}B \right),
	\end{align}
	where $\frac{1}{p}+ \frac{1}{q} = 1$.  Moreover, they also showed that this definition of the $L_p$ combination agrees with the original one defined by Firey for $A,B \in \cvxb_{(o)}$.  Moreover, Lutwak in \cite{Lutwak1,Lutwak2} developed a deep study of the $L_p$-Brunn-Minkowski theory which parallels and generalizes the traditional Brunn-Minkowski theory in essence. 		In particular,  for a convex body $K\in\cvxb_{(o)}$,  the Kubota's integral formula expresses the quermassintegral \(W_{j}(K)\) for $j\in\{0,1,\cdots,n-1\}$ as 
	$$
	W_{j}(K)=c_{n,j} \int_{G_{n, n-j}}\V_{n-j}(K|H) d \nu_{n, n-j}(H).
	$$
	Here $c_{n,j}=\frac{\omega_{n}}{\omega_{n-j}}$, $K|H$ is the
	projections of \(K\) on the 	\((n-j)\)-dimensional hyperplane $H$ belonging to the  Grassmannian manifold \(G_{n, n-j}\)---the
	$(n-j)$-dimensional subspaces of \(\mathbb{R}^{n}\)  equipped with the Haar probability measure \(\nu_{n, n-j}\). In \cite{Lutwak1}, the \emph{mixed $p$-quermassintegrals} of two convex bodies $K,L\in\cvxb_{(o)}$ is defined naturally as the variation formula of $W_j$ with respect to the $L_p$ Minkowski sum for convex bodies, i.e.,
	\begin{equation}\label{variationconvexbody}
	W_{p,j}(K,L)  = \frac{p}{n-j} \cdot \frac{d}{d\e}W_j(K+_p\e \cdot_p L) \Big\rvert_{\e=0}= \frac{1}{n-j} \int_{\s^{n-1}} h_L(u)^ph_K(u)^{1-p}dS_j(K,u), 
	\end{equation}
	where $S_j(K,\cdot)$ is the $j$-th surface area measure for $K$ defined on $S^{n-1}.$ If $j=0$, it recovers the classical $L_p$ mixed volume for convex bodies, and $S_0(K,\cdot)=S(K,\cdot)$ is the surface area measure on $S^{n-1}.$
	
	The Brunn-Minkowski theory has parallel ``liftings" to the theory of functions through the convex indicator function  $I_K \colon \R^n \to \R_+ \cup \{+\infty\}$, characteristic function $\chi_K(x)$ associated to $K \in \cvxb_{(o)}$ and many others, see references \cite{AAM,AlGMJV,AFO,AKM,Ball,Co,Co1,Co2,CLM2,Ro} and measures \cite{AHNRZ,Bobkov,GZ,KlartagMilman,Liv,LMNZ,arno,arno1,PT,P,JesusManuel}, etc. 
	
	One similar parallel definition for  ``Minkowski summation'' of convex bodies for  functions---\emph{supremal-convolution}  is defined as follows.
	For more information please see references for example \cite{AGM,BORELL,BrascampLieb,Co2}. 
	
	\subsection{Supremal-convolution  for functions}
	To begin with, given $s\in [-\infty,\infty],$ $a,b \geq 0$,  the $s$-mean of $a$ and $b$ with respect to nonnegative coefficients $\alpha, \beta\geq 0$ is denoted as 
	\[
	M_s^{(\alpha,\beta)}(a,b) = \begin{cases}
	\left(\alpha a^s + \beta b^s \right)^{\frac{1}{s}}, &\text{if } s\neq 0, \pm \infty,\\
	a^{\alpha}b^{\beta}, &\text{if } s= 0,\\
	\max\{a,b\}, &\text{if } s= +\infty,\\
	\min\{a,b\} &\text{if } s= - \infty,
	\end{cases}
	\]
	whenever $ab>0$, and  $M_s^{(\alpha,\beta)}(a,b) = 0$ otherwise.   A measure $\mu$ on $\R^n$ is $s$-concave if, for any Borel sets $A,B \subset \R^n$ and any $t \in [0,1]$, one has that 
	$
	\mu((1-t)A+tB)\geq M_s^{((1-t),t)}(\mu(A),\mu(B));
	$
	and  a measure $\mu$ on $\R^n$ is $\log$-concave (when $s=0$) if, for any Borel sets $A,B \subset \R^n$ and any $t \in [0,1]$, one has 
	$
	\log(\mu((1-t)A+tB)) \geq (1-t) \log(\mu(A)) + t \log(\mu(B)),
	$
	or equivalently, 
	$
	\mu((1-t) A + t B) \geq \mu(A)^{1-t}\mu(B)^t. 
	$

	A function $f \colon \R^n \to \R_+$ is $s$-concave if, for all $x,y \in \R^n$ and any $t \in [0,1]$, one has that  
	$
	f((1-t)x+ty) \geq M_s^{((1-t),t)}(f(x),f(y)).
	$
	The case when $s=0$ and $s=-\infty$ are referred to as $\log$-concave and quasi-concave functions, respectively.  Note that quasi-concavity of $f$ is equivalent to the condition that the super-level sets $C_f(r): = \{x \in \R^n \colon f(x) \geq r \}$ are convex for any constant $r>0$. Moreover, any $s$-concave function with its maximum at the origin is radially decreasing.
	
	An important inequality of the Brunn-Minkowski type for functions which links $s$-concave measures (with $s$-concave density functions) is the \emph{Borell-Brascamp-Lieb inequality} (see \cite{BORELL,BrascampLieb,L,P,RS}).	
	Let $t \in [0,1]$ and $s \in [-1/n,\infty]$. Given a triple of measurable functions $h, f,g \colon \R^n \to \R_+$ satisfying the condition 
	\begin{equation}\label{conditionbbl}
	h((1-t)x+ty) \geq M_s^{((1-t),t)}(f(x),g(y))
	\end{equation}
	for any $x,y \in \R^n$, there is
	\begin{equation}\label{e:Borell-Brascamp-Lieb }
	\int_{\R^n} h(x)dx \geq M_{\frac{s}{1+ns}}^{((1-t),t)} \left(\int_{\R^n} f(x)dx, \int_{\R^n} g(x) dx \right).
	\end{equation}	
	The case when $s = 0$ is referred to as the Pr\'ekopa-Leindler inequality and was proven firstly by Leindler in \cite{L} and Pr\'ekopa in \cite{P}. 
	The minimal function satisfying the condition (\ref{conditionbbl}) of Borell-Brascamp-Lieb inequality is the supremal-convolution of the functions $f$ and $g$ (or $s$-supremal-convolution); that is, the function $m_{t,s} \colon \R^n \to \R_+$ defined by 
	\[
	m_{t,s}(z) = \sup_{z = (1-t)x+ty} M_s^{((1-t),t)}(f(x),g(y)).
	\]

	The introduction of the supremal-convolution of functions leads to the following notions of addition and scalar multiplication of functions: given $s \in [-\infty,\infty]$, $f,g \colon \R^n \to \R_+$ 
	\[
	(f \oplus_s g)(z) =  \sup_{z = x+y} \begin{cases}
	\left(f(x)^s+ g(y)^s\right)^{\frac{1}{s}},  &\text{if } s \neq 0, \pm \infty,\\
	f(x)g(y), &\text{if } s = 0,\\
	\max\{f(x),g(y)\},  &\text{if } s = +\infty,\\
	\min\{f(x),g(y)\},  &\text{if } s = -\infty,
	\end{cases}
	\]
	and for $\alpha >0$,
	\[
	\quad (\alpha \times_s f)(x) = \begin{cases} 
	\alpha^{\frac{1}{s}} f\left(\frac{x}{\alpha}\right),  &\text{if } s \neq 0, \pm \infty,\\
	f(x)^{\alpha},  &\text{if } s = 0, \\
	f(x), &\text{if}\ s = \pm \infty.
	\end{cases}
	\]
	The $s$-concavity is closed under the supremal-convolution operation; i.,e, $f \oplus_s g$ and $\alpha \times_s f$ defined above are $s$-concave whenever $f$ and $g$ are as well (see \cite[Proposition~2.1]{BCF}).  In addition, for any non-empty sets $A,B \subset \R^n$ and $\alpha, \beta >0$, we have that 
	$
	(\alpha \times_s \chi_A) \oplus_s (\beta \times_s \chi_B) = \chi_{\alpha A + \beta B}
	$
	whenever $\alpha + \beta =1$. Denote the \emph{total mass} of $f$ as 
	$
	I(f) = \int_{\R^n} f(x) dx. 
	$
	Then based on  this supremal-convolution definition, the Borell-Brascamp-Lieb inequality  \eqref{e:Borell-Brascamp-Lieb } asserts that for any $s\geq -1/n$, 
	\begin{equation}\label{bblclassic}
	I\big(((1-t) \times_s f) \oplus_s (t \times_s g)\big) \geq M_{\frac{s}{1+ns}}^{(1-t,t)}(I(f),I(g)).
	\end{equation}
	
	In \cite{RX}, the authors established
	the  \emph{$L_p$-Borell-Brascamp-Lieb  inequality} based on the  geometric extension of $L_p$ Minkowski sum with respect to measurable sets in $\R^n$ (\ref{lpsummationgeometry}) using $L_p$ coefficients of Lutwak, Yang, and Zhang in \cite{LYZ}; that is, let $p\geq1$, $1/p+1/q=1,$  $s\geq 0$ and $f,g,h \colon \R^n \to \R_+$ be a triple of bounded integrable functions.  For simplicity, we denote the \emph{$L_p$ coefficients} for $\lambda\in[0,1]$ and $t\in[0,1]$ as $$C_{p,\lambda,t} :=(1-t)^\frac{1}{p}(1-\lambda)^{\frac{1}{q}}, \quad D_{p,\lambda,t}:=t^\frac{1}{p}\lambda^{\frac{1}{q}}$$ in later context.
	Suppose, in addition, that this triple satisfies the $L_p$-Borell-Brascamp-Lieb inequality condition
	\begin{equation}\label{e:LpBBLassumption}
	h\left( C_{p,\lambda,t}x + D_{p,\lambda,t} y\right) \geq \left[C_{p,\lambda,t} f(x)^{s} + D_{p,\lambda,t} g(y)^{s}\right]^{\frac{1}{s}}
	\end{equation}
	for every $x \in \text{supp}(f)$, $y\in \text{supp}(g)$ and every $\lambda \in [0,1]$. Then the following integral inequality holds:
	\begin{equation*}\label{e:LPBBLinequality0}
	I(h) \geq M_{\frac{ps}{1+ns}}^{((1-t),t)}\left(I(f),I(g) \right).
	\end{equation*}
	Naturally, the authors in \cite{RX} gave the definition of \emph{$L_{p,s}$ supremal-convolution} of $f:\R^n\rightarrow\R_+$ and $g:\R^n\rightarrow\R_+$ for $s\geq0$ and $p\geq1$, i.e., 
	\begin{align}\label{Lpsupsum}
	[f \plp g](z) &:= \sup_{0 \leq \lambda \leq 1} \left(\sup_{z = (1-\lambda)^{\frac{1}{q}}x + \lambda^{\frac{1}{q}}y} M_s^{\left((1-\lambda)^{\frac{1}{q}},\lambda^{\frac{1}{q}} \right)}(f(x),g(y)) \right)\\
	\nonumber&=\sup_{0 \leq \lambda \leq 1}\left( [(1-\lambda)^{\frac{1}{q}} \times_s f] \oplus_s [\lambda^{\frac{1}{q}} \times_s g](z)\right),
	\end{align}
	where $1/p+1/q=1$. And given any scalar $\alpha > 0$,  the scalar multiplication $\times_{p,s}$ satisfies
	\begin{equation}\label{Lpsupproduct}
	(\alpha \times_{p,s} f)(x) = \alpha^{s/p} f\left(\frac{x}{\alpha^{1/p}} \right).
	\end{equation}
	Therefore the $L_p$-Borell-Brascamp-Lieb inequality concludes that for any $s\geq 0$, 
	\begin{equation*}\label{bblclassic1}
	I\big(((1-t) \times_{p,s} f) \oplus_{p,s} (t \times_{p,s} g)\big) \geq M_{\frac{ps}{1+ns}}^{(1-t,t)}(I(f),I(g)),
	\end{equation*}
	which is a $L_p$ generalization of formula (\ref{bblclassic}).

	Another  type of summations---\emph{Asplund summation} (or $L_1$ Asplund summation) for functions using infimal convolution ($\square$) for base functions,  and its $L_p$ extensions for log-concave functions \cite{FXY, Rotem, Rotem2} with $L_p$ averages for base functions is defined as follows. In the following, we list some basics of Asplund summation for functions first.

	\subsection{Asplund summation of $s$-concave function } Consider the following class of bounded $s$-concave functions:
	\[
	\mathcal{F}_s(\R^n) = \left\{f \colon \R^n \to \R_+, f \text{ is } s\text{-concave}, \text{u.s.c}, f \in L^1(\R^n), f(o) = \|f\|_{\infty}>0 \right\},
	\]
	where u.s.c. stands for upper
	semi-continuous. The class $\mathcal{F}_0(\R^n)$, is the class of all such $\log$-concave functions, and $\mathcal{F}_{-\infty}(\R^n)$ is the class of all such quasi-concave functions.

	To begin with, we will introduce reasonable base classes of convex functions (see\cite{AM1,AM2,Co2,Rock1,Rock2} for example). Denote the set of proper (non-empty domain) convex functions $u \colon \R^n \to \R \cup \{+\infty\}$ that are lower semi-continuous by $\text{Cvx}(\R^n)$. The infimal convolution of $u,v \in \text{Cvx}(\R^n)$ is the convex function defined by 
	\begin{equation}\label{infcon}
	(u\square v)(x) = \inf_{y\in \R^n} \{u(x-y)+v(y)\},
	\end{equation}
	which should be viewed as an addition on the class $\text{Cvx}(\R^n)$. Moreover, the  scalar multiplication satisfies
	\[
	(\alpha \times u)(x) = \alpha u(x/\alpha). 
	\]
	To understand the infimal convolution geometrically,  we can see that the function  $u \square v$ whose epigraph is the Minkowski sum of the epigraphs of $u$ and $v$ \cite{Co2,CoGoNi,ColesantiFragala}:
	\begin{equation}\label{epigraph}
	\text{epi}(u \square v) = \text{epi}(u) + \text{epi}(v),
	\end{equation}
	where 
	$
	\text{epi}(u) = \{(x,y) \in \R^n \times \R \colon y \geq u(x)\} 
	$ and $``+"$ denotes the Minkowski sum in $\R^n.$

	The classical Legendre transformation $u^* \colon \text{Cvx}(\R^n) \to \text{Cvx}(\R^n)$ is given by 
	$$
	u^*(x) = \sup_{y \in \R^n}[\langle x,y \rangle - u(y)].
	$$
	It is easy to check that $(I_K)^*= h_K$ for $K \in \cvxb_o$. 	
	For an extensive list of the properties of the Legendre transformation please see \cite{Co2,ColesantiFragala, Rinott,Rock1,Rock2} for reference. A crucial connection between the infimal convolution and Legendre transformation on the class $\text{Cvx}(\R^n)$ is 
	\begin{equation}\label{e:LegendreInfconvolution}
	((\alpha \times u) \square (\beta \times v))= (\alpha u^* + \beta v^*)^*
	\end{equation} for $\alpha, \beta\geq0$.
	
	Alternatively,  the class of \emph{super-coercive} geometric convex functions (originally considered in \cite{AM2} where a second duality transformation was discovered and classified) is defined as
	\[
	C_s(\R^n) = \left\{u\in \text{Cvx}(\R^n):\ u(o)=0, \lim_{x \to \infty} \frac{u(x)}{\|x\|} = +\infty\right\}\subset\text{Cvx}(\R^n).
	\]
	Denote 
	$
	C_s(\R^n)^* = \{u \colon \R^n \to \R_+, u \text{ is convex, proper, } u(o) = 0\},
	$
	where the class $C_s(\R^n)^*$ can be thought of as the dual space of $C_s(\R^n)$ via the Legendre transform.

	In \cite{Rotem} Rotem established a connection between members of $\mathcal{F}_s(\R^n)$ and $C_s(\R^n)$ for any $s \in [-\infty,\infty]$. Given $f \in \mathcal{F}_s(\R^n)$,  the base function for $f$ is defined as \cite[Definition 8]{Rotem}, $u_f: \mathcal{F}_s(\R^n)\rightarrow C_s(\R^n)$ such that 
	\[
	f(x) = \left(1 - su_f(x)\right)_+^{\frac{1}{s}},
	\]
	where $a_+= \max\{a,0\}$. 
	When $s=0$, $f(x)=e^{-u_f(x)}.$
	In particular, for $f = \chi_K$ for some $K \in \cvxb_{(o)}$, $u_{f} = I_K$. 
	In \cite{Rotem} the following operations--\emph{Asplund summation} $\star_s$ and $ \cdot_s$ for $s$-concave functions were considered: given $f,g \in \mathcal{F}_s(\R^n)$ and $\alpha > 0$,
	\[
	\quad u_{f \star_s g}(x) = (u_f \square v_g)(x)\quad \text{ and }\quad u_{\alpha \cdot_s f}(x) = \alpha u_f\left(\frac{x}{\alpha}\right).
	\]
	In particular, \cite[Proposition~10]{Rotem} asserts that, for any $t \in [0,1]$ and $f,g \in \mathcal{F}_s(\R^n)$, the supremal-convolution coincides with the Asplund summation with coefficients $((1-t),t)$; that is
	\begin{equation}\label{compare}
	[((1-t) \cdot_s f) \star_s (t \cdot_s g)]=[1-su_{((1-t) \cdot_s f) \star_s (t \cdot_s g)}]_+^{1/s}= [((1-t) \times_s f) \oplus_s (t \times_s g)];
	\end{equation}
	or equivalently, using equality \eqref{e:LegendreInfconvolution}, the above equality can be more explicitly stated as 
	\[
	[((1-t)\times_s f) \oplus_s (t \times_s g)]= \left[1 - s((1-t) u_f^* + t v_g^*)^*\right]_+^{\frac{1}{s}}.
	\]
	For $u\in C_s(\R^n)$, consider the integral operator $J_s \colon C_s(\R^n) \to \R_+$ defined by 
	\[
	J_s(u) = \int_{\R^n} \left[1-su(x) \right]_+^{\frac{1}{s}} dx.
	\]
	Then the Borell-Brascamp-Lieb inequality implies that, for any $s\geq -1/n$ and $u,v \in C_s(\R^n)$, one has that 
	\[
	J_s([((1-t) \times u)\square (t \times v)]) \geq M_{\frac{s}{1+ns}}^{((1-t),t)}(J_s(u),J_s(v)). 
	\]

	In \cite{FXY} and \cite{Rotem2}, the authors proposed the $L_p$ summations in terms of the base functions for $s=0$ (log-concave functions), and the solutions to the corresponding Minkowski type problems for $p\geq 1$ and $0<p<1$ are also proposed and solved, respectively. Based on these two types of summations  for functions above, i.e., the $L_p$ supremal-convolution for $p\geq1$, $s\geq0$ and  Asplund summation for $s$-concave functions including log-concave case, we consider more complicated cases of summations for $s$-concave functions for various cases for $p$ and $s$. 
	\subsection{Main results}
	Our paper mainly focus on $L_p$ functional theory which naturally extending $L_p$ Brunn-Minkowski theory for convex bodies (measurable sets \cite{LYZ}) in the geometric setting.
	These include two types of $L_p$ additions for functions,  the $L_p$-Borell-Brascamp-Lieb type inequalities, and the $L_{p,s}$ concavity for functions and measures, and the corresponding  variation formula in terms of the new defined $L_p$ Asplund summation (perturbation) for $s$-concave functions,  etc. Particularly, 	
	Section \ref{section2} focuses on detailed definitions of the $L_p$ sum for functions, such as for $s$-concave functions for $p\geq1$ and $0<p<1$, respectively. In summary, we introduce the following  new definitions 
	\begin{enumerate}
		\item $p\geq 1$ and $s\in[-\infty,\infty]$, the $L_{p,s}$ sup-convolution,

		\item $0< p< 1$ and $s\in[-\infty,\infty]$, the $L_{p,s}$ inf-sup-convolution,

		\item $p\geq 1$ and $s\in(-\infty,\infty)$, $L_{p,s}$ Asplund summation for $s$-concave functions,

		\item $0<p<1$ and $s\in(-\infty,\infty)$, $L_{p,s}$ Asplund summation for $s$-concave functions.
		
	\end{enumerate}
	More in detail, we extend the $L_{p,s}$ sup-convolution from $s\in[0,\infty]$ to $s\in[-\infty,\infty]$  in (\ref{Lpsupsum}) and (\ref{Lpsupproduct}), and analyze the corresponding properties for $p\ge 1$.	
	Based on the definition of $L_p$ Minkowski sum for convex bodies for $0<p<1$ using Wulff shapes (or Aleksandrov bodies), we give a functional version for the summation---\emph{$L_{p,s}$ inf-sup-convolution} accordingly.
	\begin{def}\label{d:infsupconvolution} Let $0<p<1$, $1/p+1/q=1$ and $s \in [-\infty,\infty]$.  Given Borel measurable functions $f,g \colon \R^n \to \R_+$, $o\in int (supp(f))$, $o\in int(supp(g))$ and $\alpha,\beta >0$, we define the $L_{p,s}$ $\inf$-$\sup$-convolution of $f$ and $g$ based on (\ref{lpsummationgeometry}) (replace ``$\sup$" to ``$\inf$") as
		\begin{equation*}\label{e:infsupconvolution}
		[\alpha \times_{p,s} f \oplus_{p,s} \beta \times_{p,s}g ](z) = \inf_{0 \leq \lambda\leq 1} \left[\sup_{z =\alpha^{\frac{1}{p}} (1-\lambda)^{\frac{1}{q}}x +\beta^{\frac{1}{p}} \lambda^{\frac{1}{q}}y} M_s^{\left((1-\lambda)^{\frac{1}{q}},\lambda^{\frac{1}{q}} \right)}\left(\alpha^{\frac{1}{sp}}f(x),\beta^{\frac{1}{sp}} g(y)\right) \right].
		\end{equation*}		
	\end{def}
	Elementary properties are also provided by a detailed analysis for this new sum in Proposition~\ref{t:properties22}. 
	
	Moreover, following the method of $L_p$ Asplund summation for log-concave functions when $p\geq 1$ \cite{FXY} and $0<p<1$ \cite{Rotem2}, we introduce the $L_{p,s}$ Asplund summation for $s$-concave functions using $L_p$ addition for base functions. 
	Let $p>0$.
	Given $\alpha, \beta\geq0$ and $u,v \in C_s(\R^n)$,  the $L_{p}$ additions of $u,v$ (base functions), a generalization of (\ref{e:LegendreInfconvolution}), is
	\begin{equation*}
	[(\alpha \boxtimes_{p} u) \boxplus_p (\beta \boxtimes_p v)](x):=\{ (\alpha (u^*(x))^p + \beta (v^*(x))^p)^{1/p}\}^*
	\end{equation*}
	and the $L_{p,s}$ Asplund summation for functions in $\mathcal{F}_s(\R^n)$ is defined as follows.
	(i)	For  $p\geq 1$, $s\in (-\infty,\infty)$, given $f,g \in \mathcal{F}_s(\R^n)$, we define the $L_{p,s}$ Asplund summation with  weights $\alpha,\beta\geq0$ as
	\[
	(\alpha \cdot_{p,s} f) \star_{p,s} (\beta \cdot_{p,s} g) := \Big(1-s \big[(\alpha \boxtimes_{p} u_f) \boxplus_p (\beta \boxtimes_p v_g)\big] \Big)_+^{\frac{1}{s}}.
	\]
	(ii)	For $0<p<1$, $s\in (-\infty,\infty)$, given $f,g \in \mathcal{F}_s(\R^n)$ with $h_{f},h_{g}\ge0$, we define
	the $L_{p,s}$ Asplund summation $\alpha\cdot_{p,s}f\star_{p,s}\beta\cdot_{p,s} g$  with weights $\alpha,\beta\geq0$ as
	\[
	\alpha\cdot_{p,s}f\star_{p,s}\beta\cdot_{p,s} g:=A\left[\left(\alpha h_{f}^{p}+\beta h_{g}^{p}\right)^{1/p}\right]_s,
	\]
	where $h_f$ is the support function of $f$ and $A[\cdot]_s$ denotes the $s$-Aleksandrov function. See detailed definitions of $h_f$ and $A[\cdot]_s$ for explanation in Section \ref{section2}.

	Inspired by (\ref{compare}) we verify that when $p\geq 1$ and $s=0$, the $L_{p,s}$
	supremal-convolution agrees with the $L_{p,s}$ Asplund summation through base functions. However, for $s\neq0$,  these two summations differs with each other as the relation only works for inequalities.
	
	For these two types of $L_p$ summations for functions, it is much difficult to obtain the variation formula for $L_{p,s}$ sup-convolution for $p\geq1$ and  $L_{p,s}$ inf-sup-convolution for $0<p<1$ with delicate $L_p$ coefficients. Instead, we focus on the corresponding $L_p$-Borell-Brascamp-Lieb type inequalities in different circumstances in Section \ref{section3}. Specially, we  study new $L_p$-Borell-Brascamp-Lieb type  inequalities for functions extending works in \cite{BCF} to the $L_p$ versions  and generalizing the result in \cite{RX} from $s\geq0$ to $s\in[-\infty, \infty]$ in different methods. 
	In detail, our main goals are to solve	
	\begin{enumerate}

		\item General $L_p$-Borell-Brascamp-Lieb  inequality in terms of $\Omega$-total mass.
		\item Proof of $L_p$-Borell-Brascamp-Lieb  inequality  using mass transportation.
		\item Proof of $L_p$-Borell-Brascamp-Lieb  inequality  using classical Borell-Brascamp-Lieb inequality.
	\end{enumerate}
	
	Our first main result generalizes a theorem of Bobokov, Colesanti and Fragal\'a \cite[Theorem~4.2]{BCF} and we obtain the following theorem.
	\begin{theorem}\label{t:LpBorell-Brascamp-Lieb functionals} Let $\Omega:\mathscr{B}\rightarrow\R_+$ be $\alpha$-concave. Let $p,q \in[1,\infty]$ be such that  $1/p+1/q=1$. Let $\alpha \in [-1,+\infty]$ and $\gamma \in [-\alpha, \infty)$. Suppose that $h,f,g \colon \R^n \to \R_+$ are a triple of integrable Borel measurable (respectively, quasi-concave functions) that satisfy the condition 
		\begin{equation}\label{e:mainassumption12}
		\begin{split}
		h\left(C_{p,\lambda,t}x + D_{p,\lambda,t}y  \right) \geq\left(C_{p,\lambda,t}f(x)^{\alpha} + D_{p,\lambda,t} g(y)^{\alpha} \right)^{\frac{1}{\alpha}}
		\end{split}
		\end{equation}
		for every $x \in \text{supp}(f)$, $y \in \text{supp}(g)$, and $\lambda \in (0,1)$ whenever $f(x)g(y) >0$. 
		Then the following inequality holds:
		\begin{equation*}\label{e:functionalBorell-Brascamp-Lieb conclusion}
		\begin{split}
		\widetilde{\Omega}(h) \geq \left[(1-t) \widetilde{\Omega}(f)^{\be} + t \widetilde{\Omega}(g) ^{\be}\right]^{\frac{1}{\beta}}, \quad \be = \frac{p\alpha\gamma}{\alpha+\gamma}.
		\end{split}
		\end{equation*}	
	\end{theorem}
	Please see definitions for $\mathscr{B}$ and $\widetilde{\Omega}$ in Subsection \ref{subsection31} for details.
	Moreover, our results extend the $L_p$-Borell-Brascamp-Lieb inequality originally appearing in \cite{RX} for the case $s \geq 0$ and  Borell-Brascamp-Lieb inequality for $s\leq -1/n$ in \cite{DancsUhrin} stated as:
	\begin{theorem}
		Let $p\geq1$, $-\infty<s < \infty$, and $t \in (0,1)$. Let $f,g,h \colon \R^n \to \R_+$ be a triple of bounded integrable functions. Suppose, in addition, that this triple satisfies the condition
		\begin{equation*}
		h\left(C_{p,\lambda,t}x + D_{p,\lambda,t}y\right) \geq \left[C_{p,\lambda,t} f(x)^{s} + D_{p,\lambda,t}g(y)^{s}\right]^{\frac{1}{s}}
		\end{equation*}
		for every $x \in \text{supp}(f)$, $y\in \text{supp}(g)$ and every $\lambda \in [0,1]$. Then the following integral inequality holds:
		\begin{equation*}
		I(h) \geq
		\begin{cases} 
		M_{\gamma_1}^{((1-t),t)}\left(I(f), I(g)\right), &\text{if } s \geq -\frac{1}{n},\\
		\min \left\{\left[C_{p,\lambda,t}\right]^{\frac{1}{\gamma}}I(f), \left[D_{p,\lambda,t}\right]^{\frac{1}{\gamma}}I(g) \right), &\text{if } s \leq - \frac{1}{n},
		\end{cases}
		\end{equation*}
		for $0\leq \lambda\leq1$,	where $\gamma_1= p\gamma$ and $\gamma= \frac{s}{1+ns}$.
	\end{theorem}
	
	Based on conditions these $L_p$-Borell-Brascamp-Lieb  type inequalities, Section \ref{section4} focuses on new definitions of $L_{p,s}$ concavity for functions and measures. One typical example with respect to measure we list here is 
	\begin{definition}\label{t:lpmeasures}
		Let $p \geq 1$, $1/p + 1/q = 1$, $t \in [0,1]$, and $s \in [-\infty,+\infty]$. We say that a non-negative measure $\mu$ on $\R^n$ is $L_{p,s}$-concave if, for any pair of Borel measurable sets $A,B \subset \R^n$, one has 
		\begin{equation*}\label{e:lpmeasureconditional}
		\mu\left(C_{p,\lambda,t}A+D_{p,\lambda,t}B\right) \geq M_s^{(C_{p,\lambda,t},D_{p,\lambda,t})}(\mu(A),\mu(B))
		\end{equation*}
		for every $\lambda \in [0,1]$ and $t\in[0,1]$.  Similarly, if $s =-\infty$, the measure $\mu$ is said to be $L_{p,s}$-quasi-concave, and if $s =0$, the measure $\mu$ is said to be $L_{p,s}$-log-concave.

	\end{definition}

	On the other hand, it is more reasonable to research on the variation formula for $L_{p,s}$ Asplund summation using base functions with linear coefficients. Therefore Section \ref{section5} concentrates on the  definition of $L_{p,s}$ mixed quermassintegral of functions equivalent to the derivative of quermassintegral for functions  which is similar to the theory of convex bodies by Lutwak in \cite{Lutwak1}. The main works we finish are proposing and proving:
	\begin{enumerate}
		\item Projection for functions and corresponding properties related to $L_p$ summations;
		
		\item Integral representation for $L_{p,s}$ mixed quermassintegral for functions.
	\end{enumerate}
	
	By the definition of projection of functions and analyzing the properties of the projections for functions with respect to $L_{p,s}$ convolutions in Subsection \ref{subsection51} and $L_{p,s}$ Asplund summation Subsection \ref{subsection52}  in certain circumstances, we provide the definition of quermassintegral for functions as well as the variation formula---the $L_{p,s}$ mixed quermassintegral for functions in $\mathcal{F}_s(\R^n)$ in Subsection \ref{subsection53}. That is,
	for $j\in\{0,\cdots,n-1\},$ the $j$-th \textit{quermassintegral} of function $f$, is defined as
	$$
	W_j(f):= c_{n,j} \int \limits_{G_{n,n-j}}\, \int \limits_H P_H f(x) dx\, d\nu_{n,n-j}(H)
	$$
	and the $L_{p,s}$ mixed quermassintegral for $s$-concave functions $f,g\in\mathcal{F}_s(\R^n)$ has the definition of 
	$$
	W_{p,j}^s(f,g):= \lim \limits_{\varepsilon \to 0}  \frac{W_j(f \star_{p,s}\varepsilon \cdot_{p,s} g) - W_j(f)}{\varepsilon}.
	$$
	
Through the process of finding the variation formula for the general $\Omega$-$j$th-quermassintegral  in terms of the base functions $u\boxplus_{p,s}\varepsilon \boxtimes_{p,s}v$, and thus the $\Omega$-$L_{p,s}$ mixed quermassintegral of $u,v\in C_s(\R^n)$, $\mathbb{W}_{p,j}^s(u,v),$ where $u, v$ denote the base functions for $f$ and $g$ correspondingly, we obtain the integral representation formula with respect to the $L_{p,s}$ mixed quermassintegral for $f,g\in\mathcal{F}_s(\R^n)$. That is, $p \geq 1$, $j\in \{0,\dots, n-1\}$, and $s \in \left(-\infty,\infty\right)$, let $f=(1-su)_+^{1/s}, g=(1-sv)_+^{1/s}$ such that $u,v\in C_s(\R^n)$ and $u \in C^{2,+}(\R^n)$, and $\psi \in C_c^{\infty}(\R^n)$ with $\psi = v^*$. Then the $L_{p,s}$ mixed quermassintegral for $f,g\in\mathcal{F}_s(\R^n)$ has the following integral representation: 
		\[
		W_{p,j}^s(f,g) =\frac{1}{n-j}\int_{\R^n} \frac{\left[1-su_H (x) \right]_+^{\frac{1}{s}-1} \psi_H(\nabla u_H(x))^p}{\|x\|^{j}}\varphi_H(\nabla u_H(x))^{1-p}dx
		\]
		For $s=0$, the above becomes
		\[
		W_{p,j}^0(f,g) =\frac{1}{n-j}\int_{\R^n} \frac{e^{-u_H(x)}\psi_H(\nabla u_H(x))^p \varphi_H(\nabla u_H(x))^{1-p}}{\|x\|^{j}}dx.
		\]
	(Please see definitions for $\mathbb{W}_{p,j}^s$, $u_H$ and $\psi_H$ in Subsection \ref{subsection53} for detailed information.)
	When $j=0$ and $s=0$, it recovers the integral interpretation of variation formulas in \cite{FXY} and \cite{Rotem2}, for $p\geq1$ and $0<p<1$, respectively.

	\section{Functional $L_p$ operations  for $p>0$}\label{section2}
	
	In this section, we will first extend  the original  definitions for $L_{p,s}$ supremal-convolution for  functions in \cite{RX} for $s\geq0$ to all $s\in[-\infty,\infty]$ and $p\geq 1$.
	For   $0<p<1$, we propose a brand new definition of $L_{p,s}$ summation for  functions---the $L_{p,s}$ inf-sup-convolution in Subsection \ref{subsection21}. We verify that these $L_{p,s}$ convolutions satisfy elegant properties by $L_p$ coefficients $(C_{p,\lambda,t}, D_{p,\lambda,t})$.  In Subsection \ref{subsection22}, we introduce the $L_{p,s}$ Asplund summation through the base functions for $s$-concave functions inspired by the case of log-concave functions \cite{Rotem, Rotem2, FXY} for $p\geq1$ and $0<p<1$, respectively. Furthermore in Subsection \ref{subsection23}, we compare definitions proposed in Subsections \ref{subsection21} and \ref{subsection22}, and prove that  for log-concave functions ($s=0$),  these two summations for $p\geq1$ are equivalent to each other.

	\subsection{General  $L_{p,s}$ supremal-convolution for $p>0$}\label{subsection21}
	The focus on this section is to highlight functional operations of addition and scalar multiplication which generalize the supremal-convolution $\oplus_s$ and $\times_s$ to the $L_{p,s}$ setting and returns to the $L_p$ Minkowski combination for convex bodies in the geometric setting for well selected functions originally discovered in \cite{RX}.
	
	\textbf{(i) General  $L_{p,s}$ supremal-convolution for $p\geq 1$.}
	Firstly, we extend the range of $s\in[0,\infty]$ in \cite{RX} to more general setting $s\in[-\infty,\infty]$ without changing the original formulas for $p\geq1$; that is:
	
	\begin{definition}\label{t:wholepoperations} Let $p\geq1,$ $\frac{1}{p}+ \frac{1}{q} =1 $, and $s \in [-\infty,\infty]$. Let $f,g \colon \R^n \to \R_+$ be measurable functions. We define the $L_{p,s}$ supremal-convolution of $f$ and $g$ by 
		\begin{align}
		[f \plp g](z) \nonumber&:= \sup_{0 \leq \lambda \leq 1} \left(\sup_{z = (1-\lambda)^{\frac{1}{q}}x + \lambda^{\frac{1}{q}}y} M_s^{\left((1-\lambda)^{\frac{1}{q}},\lambda^{\frac{1}{q}} \right)}(f(x),g(y)) \right)\\
		\label{supremalsum}&= \sup_{0 \leq \lambda \leq 1}\left( [(1-\lambda)^{\frac{1}{q}} \times_s f] \oplus_s [\lambda^{\frac{1}{q}} \times_s g](z)\right).
		\end{align}
		Moreover, a scalar multiplication is defined by for $\alpha\geq0$,
		\begin{equation}\label{supremalproduct}
		(\alpha \tlp f)(x) := \sup_{\tau \in [0,1]} \left[\left(\alpha^{\frac{1}{p}} \tau^{\frac{1}{q}}  \times_s f\right)(x) \right]=\alpha^{\frac{s}{p}} f\left(\frac{x}{\alpha^{1/p}}\right),
		\end{equation}
		where we set explicitly 
		\[
		(0 \tlp f)(x) = \chi_{\{0\}}(x). 
		\]
		More generally, for $\alpha, \beta\geq0$, the  $L_{p,s}$ supremal-convolution of the functions $f$ and $g$ with respect to $\alpha$ and $\beta$ is denoted as
		\[
		[\alpha \tlp f] \plp [\beta \tlp g].
		\]
	\end{definition} 
	
	Heuristically, $[\alpha \tlp f] \plp [\beta \tlp g]$ should be understood as evaluating averages of functions over the $L_p$ Minkowski combination of the supports of $f$ and $g$, that is, over the set 
	$
	\alpha \cdot_p \supp(f) +_p \beta \cdot_p \supp(g)
	$ in (\ref{lpsummationgeometry}). We illustrate the following example on how the functional operations $\tlp, \plp$ naturally extend $\cdot_p,+_p$ in the geometric background for $p\geq1$ .  
	\begin{example}
		Suppose that $p\geq1,$ $s \in [-\infty,\infty]$ and $t \in (0,1)$. Let $1/p+1/q=1$, and let $f = \chi_A$ and $g= \chi_B$ be characteristic functions of Borel sets $A,B \subset \R^n$, respectively, and set
		\[
		h_{p,t,s} = ((1-t) \tlp f) \plp (t \tlp g). 
		\]
		Then 
		$
		h_{p,t,s} = \chi_{(1-t) \cdot_p A +_p t \cdot_p B}.
		$
	\end{example}
	
	As the above example shows, there's a natural embedding on the class of Borel measurable sets equipped with the operations $\cdot_p,+_p$ in (\ref{lpsummationgeometry}) into the family of measurable functions equipped with the operations $\tlp,\plp$ in (\ref{supremalsum}) and (\ref{supremalproduct}), respectively for $p\geq1$.  
	
	It was shown in \cite{RX} that $[(1-t) \times_{p,s} f ] \plp [t \times_{p,s} g]\in \mathcal{F}_s(\R^n)$ whenever $f,g \in \mathcal{F}_s(\R^n)$ for $s\in[-\infty, \infty].$	Except for these properties, the next proposition concerns some key properties of the operations $\tlp$ and $\plp$.

	\begin{proposition}\label{t:properties1} Let $f,g,h	 \colon \R^n \to \R_+$ be arbitrary, not identically zero, functions defined on $\R^n$, let $s \in [-\infty,\infty]$, $p\geq1$, and $\alpha, \beta, \gamma >0$.  Then the following hold: 
		\begin{enumerate}[(a)]
			
			\item Homogeneity: 
			\[
			(\alpha \tlp f) \plp (\beta \tlp g) = (\alpha+\beta) \tlp \left[\left(\frac{\alpha}{\alpha + \beta} \tlp f \right) \plp\left( \frac{\beta}{\alpha + \beta} \tlp g\right)\right]
			\]
			for  $s \neq \pm \infty$.
			
			\item Measurability: $(\alpha \tlp f) \plp (\beta \tlp g)$ is measurable whenever both $f$ and $g$ are Borel measurable.

			\item  Commutativity:
			$(\alpha\tlp f) \plp (\beta\tlp g) = (\beta\tlp g) \plp(\alpha\tlp f)$.

		\end{enumerate}
		
	\end{proposition}
	\begin{remark}
		Here by definitions of $\oplus_{p,s}$ and $\times_{p,s}$, we can show that
		$$[(\alpha \tlp f ) \plp (\beta\tlp g)]  \plp (\gamma\tlp h) \neq
		(\alpha \tlp f)  \plp [(\beta\tlp g)  \plp (\gamma\tlp h)]$$ by the definition of $L_{p,s}$ supremal-convolution while when $p=1$ the equality holds in \cite{BCF}. The core difference is the complex coefficients in $L_p$ case $(C_{p,\lambda,t}, D_{p,\lambda,t})$ leading to delicate computation for $p\geq 1$.
		
	\end{remark}
	
	\begin{proof}[Proof of Proposition \ref{t:properties1}]
		We give a detailed proof of (a)-(b) following similar steps of the case $p=1$ in \cite{BCF} and (c) is omitted for simiplicity. 
		For $(a)$, we assume that $s \in \R \setminus \{0\}$ (the other cases follow by continuity), $p\geq1$ and  $1/p+1/q=1$. Observe, when $\bar{z} = \frac{z}{(\alpha+\beta)^{\frac{1}{p}}}$, we have 
		\begin{align*}
		&(\alpha+\beta) \tlp\left[\left(\frac{\alpha}{\alpha + \beta} \tlp f \right) \plp\left( \frac{\beta}{\alpha + \beta} \tlp g\right)\right](z)\\
		&=(\alpha+\beta)^{\frac{s}{p}}\left[\left(\frac{\alpha}{\alpha + \beta} \tlp f \right) \plp\left( \frac{\beta}{\alpha + \beta} \tlp g\right)\right](\bar{z})\\
		&=(\alpha+\beta)^{\frac{s}{p}}\sup_{0\leq\lambda\leq1} \left\{\sup_{\bar{z}= \left(\frac{\alpha}{\alpha+\beta}\right)^{\frac{1}{p}}(1-\lambda)^{\frac{1}{q}}x+\left(\frac{\beta}{\alpha+\beta}\right)^{\frac{1}{p}}\lambda^{\frac{1}{q}}y}M_{s}^{\left( \left(\frac{\alpha}{\alpha+\beta}\right)^{\frac{1}{p}}(1-\lambda)^{\frac{1}{q}}, \left(\frac{\beta}{\alpha+\beta}\right)^{\frac{1}{p}}\lambda^{\frac{1}{q}} \right)}(f(x),g(y))\right\}\\
		&=\sup_{0\leq\lambda\leq1} \left\{\sup_{z= \alpha^{\frac{1}{p}}(1-\lambda)^{\frac{1}{q}}x+\beta^{\frac{1}{p}}\lambda^{\frac{1}{q}}y}M_{s}^{\left( \alpha^{\frac{1}{p}}(1-\lambda)^{\frac{1}{q}}, \beta^{\frac{1}{p}}\lambda^{\frac{1}{q}} \right)}(f(x),g(y))\right\}\\
		&= [(\alpha \tlp f) \plp (\beta \tlp g)](z),
		\end{align*}
		as desired. 
		
		For (b), suppose that $f,g$ are Borel measurable functions. Let $a >0$ and set $$h(z):=[(\alpha \tlp f) \plp (\beta \tlp g)](z).$$ We need to show that the level set 
		$
		\{x \in \R^n \colon h(x) < a\}
		$
		is measurable for any fixed constant $a>0$. Observe that by (\ref{supremalsum}) and (\ref{supremalproduct}),
		\begin{eqnarray*}
			\{z \in \R^n \colon h(z) < a\} &=& \{z \in \R^n \colon [(\alpha \tlp f) \plp (\beta \tlp g)](z) < a \}\\&=&
			\left\{ z \in \R^n \colon \sup_{0 \leq \lambda \leq 1} \left[\alpha^{\frac{1}{p}}(1-\lambda)^{\frac{1}{q}} \times_s f \oplus_s \beta^{\frac{1}{p}}\lambda^{\frac{1}{q}} \times_s g \right](z) <a \right\}\\&=&
			\bigcap_{\lambda \in [0,1] \cap \Q} \left\{z \in \R^n \colon \left[\alpha^{\frac{1}{p}}(1-\lambda)^{\frac{1}{q}} \times_s f \oplus_s \beta^{\frac{1}{p}}\lambda^{\frac{1}{q}} \times_s g \right](z) <a \right\},
		\end{eqnarray*}
		where $\Q$ denotes all rational  numbers in $\R.$

		It follows from the fact  \cite[Page 139]{BCF} that the functions of the form (case for $p=1$)
		\[
		\alpha \times_s f \oplus_s \beta \times_s g
		\]
		are measurable whenever $f$ and $g$ are Borel measurable, and  the countable intersection of measurable sets remains measurable, as desired.
		
	\end{proof}
	We can see from above that the $L_p$ coefficients $(C_{p,\lambda,t},D_{p,\lambda,t})$  are well defined and have elegant properties. In fact, we can see that it has a close relation with the $p$-mean of parameters in the following lemma. Recall that 
	\[
	C_{p,\lambda,t}=(1-t)^{1/p}(1-\lambda)^{1/q}, \quad D_{p,\lambda,t}=t^{1/p}\lambda^{1/q}.
	\]
	\begin{lemma}\label{supremallp}
		Let $a,b\geq0$.
		
		\begin{enumerate}
			\item Let $p\geq1.$ For $t\in[0,1]$, we have		
			$$\sup_{0\leq\lambda\leq1} [C_{p,\lambda,t}a+D_{p,\lambda,t}b]=((1-t)a^p+tb^p)^{1/p};$$ 	
			\item Let $p<0.$ For $t\in(0,1)$, we have		
			$$\sup_{0\leq\lambda\leq1} [C_{p,\lambda,t}a+D_{p,\lambda,t}b]=((1-t)a^p+tb^p)^{1/p};$$ 	
			
			\item  Let $0<p<1$.   For  $t\in(0,1)$, we have
			$$\inf_{0\leq\lambda\leq1}[ C_{p,\lambda,t}a+D_{p,\lambda,t}b]=((1-t)a^p+tb^p)^{1/p}.$$
		\end{enumerate}
		
	\end{lemma}
	\begin{proof}
		Consider the function $$F(\lambda):=C_{p,\lambda,t}a+D_{p,\lambda,t}b.$$
		Observe that $F$ is concave for $p\geq1$ and $p<0$ with maximum $((1-t)a^p+tb^p)^{1/p}$, and that $F$ is convex for $0<p<1$ with the same formula for minimum. Therefore, we obtain the $p$-mean values on the right hand side of the equalities.	
	\end{proof}
	Note that we replace $1-t$ and $t$ by general coefficients $\alpha>0$ and $\beta>0$, then similar results holds naturally.
	
	\textbf{ (ii) $L_{p,s}$ inf-sup-convolution for  $0<p<1.$}
	In the following, we address an extension on the $L_p$ convolution when $p \in (0,1)$ under the inspiration of Lemma \ref{supremallp}. To begin with, we recall that in \cite{LYZ} the authors extended the definition of the $L_p$ Minkowski combinations due  to Firey to the case of $p \in (0,\infty]$. They considered, for convex bodies $K,L \in \mathcal{K}^n_{(o)}$ and scalars $\alpha, \beta>0$, the Wulff shape \cite{BLYZ} given by \begin{equation}\label{e:wulffshapedefinition}
	\alpha \cdot_p K +_p \beta \cdot_p L = \bigcap_{u \in \s^{n-1}} \left\{x \in \R^n \colon \langle x, u \rangle \leq \left(\alpha h_K(u)^p + \beta h_L(u)^p \right)^{\frac{1}{p}} \right\}=[\left(\alpha h_K(u)^p + \beta h_L(u)^p \right)^{\frac{1}{p}}],
	\end{equation}
	where the Wulff shape of a function $f\in C^+(S^{n-1})$ is
	\[
	[f]=\bigcap_{u \in \s^{n-1}}\left\{x\in\R^n: \langle x, u \rangle \leq f(u)\right\}.
	\]

	It is clear that the above definition is equivalent to Firey's original definition in the case $p \geq 1$.  Here we present an analogue of the  definition \eqref{e:wulffshapedefinition}  to general non-empty Borel sets in $\R^n$ as follows with $0<p<1$. That is, 
	
	\begin{definition}
		For $p\in (0,1)$, Borel sets $A,B$ each having the origin as an interior point, and scalars $\alpha, \beta >0$, we define $L_p$ summation for $A$ and $B$ as
		\begin{equation}\label{e:generalfractionalpdefinition}
		\alpha \cdot_p A +_p \beta \cdot_p B := \bigcap_{0 \leq \lambda\leq 1} \alpha^{\frac{1}{p}} (1-\lambda)^{\frac{1}{q}}A +  \beta^{\frac{1}{p}} \lambda^{\frac{1}{q}}B,
		\end{equation}
		where $\frac{1}{p}+ \frac{1}{q} =1$. 
	\end{definition}
	It can be checked that formula~(\ref{e:generalfractionalpdefinition}) naturally glues with the definition of the $L_p$ Minkowski combination ($p\geq1$) due to Lutwak, Yang and Zhang in \cite{LYZ} when one takes $p=1$. Similar to \cite{LYZ}, we have the following result. 
	
	\begin{proposition}\label{proplessthan1}The definitions \eqref{e:wulffshapedefinition} and \eqref{e:generalfractionalpdefinition} coincide on the class $\mathcal{K}^n_{(o)}$ for $0<p<1$.
	\end{proposition}
	
	\begin{proof} Let $K, L \in \mathcal{K}^n_{(o)}$ and $\alpha, \beta >0$.  Then
		\begin{align*}
		&\bigcap_{0\leq\lambda\leq1} \alpha^{\frac{1}{p}} (1-\lambda)^{\frac{1}{q}}K +  \beta^{\frac{1}{p}} \lambda^{\frac{1}{q}}L\\
		& = \bigcap_{0\leq\lambda\leq1} \bigcap_{u \in \s^{n-1}} \left\{x \in \R^n \colon \langle x, u \rangle \leq  h_{\alpha^{\frac{1}{p}} (1-\lambda)^{\frac{1}{q}}K}(u) +  h_{\beta^{\frac{1}{p}} \lambda^{\frac{1}{q}}L}(u) \right\}\\
		&= \bigcap_{0\leq\lambda\leq1} \bigcap_{u \in \s^{n-1}} \left\{x \in \R^n \colon \langle x, u \rangle \leq \alpha^{\frac{1}{p}} (1-\lambda)^{\frac{1}{q}} h_{K}(u) + \beta^{\frac{1}{p}} \lambda^{\frac{1}{q}} h_{L}(u) \right\}\\
		&=\bigcap_{u \in \s^{n-1}} \left\{x \in \R^n \colon \langle x, u \rangle \leq \inf\left\{ \alpha^{\frac{1}{p}} (1-\lambda)^{\frac{1}{q}} h_{K}(u) + \beta^{\frac{1}{p}} \lambda^{\frac{1}{q}} h_{L}(u): 0\leq\lambda\leq1\right\} \right\}. 
		\end{align*}
		Using Lemma \ref{supremallp} (3), we see that 
		\[
		\inf\left\{ \alpha^{\frac{1}{p}} (1-\lambda)^{\frac{1}{q}} h_{K}(u) + \beta^{\frac{1}{p}} \lambda^{\frac{1}{q}} h_{L}(u): 0\leq\lambda\leq1\right\} = \left(\alpha h_K(u)^p + \beta h_L(u)^p \right)^{\frac{1}{p}}.
		\]
		This confirms  the assertion of this proposition. 
	\end{proof}
	
	With the formula \eqref{e:generalfractionalpdefinition} and the above proposition in hand, we are in a position to define a functional counterpart of the $L_p$ Minkowski combinition in the setting $p \in (0,1)$ that coincides with Definition~\ref{t:wholepoperations} in the case $p =1$, which we refer to as the $L_{p,s}$ $\inf$-$\sup$ convolution  for functions.

	\begin{definition}\label{d:infsupconvolution} Let $0<p<1$, $1/p+1/q=1$ and $s \in [-\infty,\infty]$.  Given Borel measurable functions $f,g \colon \R^n \to \R_+$, each having support containing the origin in their interior, we define the $L_{p,s}$ $\inf$-$\sup$-convolution of $f$ and $g$ with weights $\alpha, \beta>0$ to be 
		\begin{equation}\label{e:infsupconvolution1}
		[\alpha \times_{p,s} f \oplus_{p,s} \beta\times_{p,s}g ](z) := \inf_{0 \leq \lambda \leq 1} \left[\sup_{z =\alpha^{\frac{1}{p}} (1-\lambda)^{\frac{1}{q}}x +\beta^{\frac{1}{p}} \lambda^{\frac{1}{q}}y} M_s^{\left((1-\lambda)^{\frac{1}{q}},\lambda^{\frac{1}{q}} \right)}\left(\alpha^{\frac{1}{sp}}f(x),\beta^{\frac{1}{sp}} g(y)\right) \right].
		\end{equation}
		
	\end{definition}

	The next result concerns some critical properties of the $L_{p,s}$ $\inf$-$\sup$-convolution \eqref{e:infsupconvolution1} and the proofs are similar to those of Proposition~\ref{t:properties1}. 
	
	\begin{proposition}\label{t:properties22}
		Let $h,f,g	 \colon \R^n \to \R_+$ be Borel measurable, not identically zero with supports containing the origin in their interiors, functions defined on $\R^n$. Let $s \in [-\infty,\infty]$, $0<p<1$, and $\alpha, \beta, \gamma >0$.  Then the following hold: 
		\begin{enumerate}[(a)]
			
			\item Homogeneity: 
			\[
			(\alpha \tlp f) \plp (\beta \tlp g) = (\alpha+\beta) \tlp \left[\left(\frac{\alpha}{\alpha + \beta} \tlp f \right) \plp\left( \frac{\beta}{\alpha + \beta} \tlp g\right)\right]
			\]
			for  $s \neq \pm \infty$.
			
			\item Measurability: $(\alpha \tlp f) \plp (\beta \tlp g)$ is measurable whenever both $f$ and $g$ are Borel measurable.

			\item  Commutativity:
			$(\alpha\tlp f) \plp (\beta\tlp g) = (\beta\tlp g) \plp(\alpha\tlp f)$.

			\item When $f = \chi_A$ and $g = \chi_B$ for some pair of non-empty Borel sets $A,B \subset \R^n$ with the origin in their interior, and $t \in (0,1)$, one has that 
			\[
			[(1-t) \times_{p,s} f \oplus_{p,s} t \times_{p,s}g ]= \chi_{(1-t) \cdot_p A +_p t \cdot_p B}.	
			\]

		\end{enumerate}
		
	\end{proposition}
	\begin{proof}
		For (a)-(c), the proofs follow similar lines to Proposition \ref{t:properties1} (a)-(c) with $p\geq1$ by changing $``\sup"$ to $``\inf",$ and $``\cap"$ to $``\cup"$ correspondingly with $0<p<1$. The proof of (d) follows similar lines to Proposition \ref{proplessthan1}.

	\end{proof}	
	\begin{remark}
		Associativity doesn't holds  as  
		$[(\alpha \tlp f ) \plp (\beta\tlp g)]  \plp (\gamma\tlp h) \neq
		(\alpha \tlp f)  \plp [(\beta\tlp g)  \plp (\gamma\tlp h)]$ with the $L_p$ coefficients.
		
	\end{remark}

	\subsection{$L_{p,s}$ Asplund summation for $s$-concave functions for $p>0 $}\label{subsection22}
	Next, we present the definition of $L_{p,s}$ Asplund summation for $s$-concave functions in a similar way to \cite{FXY} ($s=0$, log-concave function) with $p\geq1$ using base functions.
	
	\textbf{(i) $L_{p,s}$ Asplund summation for $s$-concave functions for $p\geq 1$.}
	Recall the $L_p$ (or $L_{p,0}$) Asplund summations for  functions using the $L_p$ operations $\square_p$ for convex functions defined in \cite{FXY}  as follows.
	Let $p\geq 1$.
	Given $\alpha, \beta >0$ and $u,v \in C_s(\R^n)$,
	\begin{equation}\label{e:lpconvexadditionp}
	[(\alpha \boxtimes_{p} u) \boxplus_p (\beta \boxtimes_p v)](x):=\{ (\alpha (u^*(x))^p + \beta (v^*(x))^p)^{1/p}\}^*. 
	\end{equation}
	In the case $p = 0$, it becomes
	\[
	[(\alpha \boxtimes_0 u) \boxplus_0 (\beta \boxtimes_0 v)](x):= [(u^*(x))^{\alpha}(v^*(x))^{\beta}]^*. 
	\]
	Therefore, we give the $L_{p,s}$ Asplund summation for functions in $\mathcal{F}_s(\R^n)$ in the same manner.	
	\begin{definition} 
		For  $p\geq 1$, $s\in (-\infty,\infty)$, given $f,g \in \mathcal{F}_s(\R^n)$, we define the $L_{p,s}$ Asplund summation as
		\[
		(\alpha \cdot_{p,s} f) \star_{p,s} (\beta \cdot_{p,s} g) := \big(1-s [(\alpha \boxtimes_{p} u) \boxplus_p (\beta \boxtimes_p v)]\big)_+^{\frac{1}{s}},
		\]	where $u$ and $v$ are base functions for $f$ and $g$, respectively.
	\end{definition}

	It was shown in \cite[Proposition~3.2]{FXY} that $[(\alpha \boxtimes_{p} u) \boxplus_p (\beta \boxtimes_p v)] \in C_s(\R^n)$ whenever $u,v \in C_s(\R^n)$ and $p \geq 1$.  A similar definition for $p \in (0,1)$ was introduced in \cite{Rotem2}. Here we give a similar $L_{p,s}$ Asplund summation for $s$-concave functions for $0<p<1$ using the Legendre transformation for  the convex functions. 
	
	\textbf{(ii) $L_{p,s}$ Asplund summation for $s$-concave functions for $0<p<1$.} We follow the notations in \cite{Rotem2} and define  the support function for $f=(1-su_f)_+^{1/s}\in\mathcal{F}_s(\R^n)$ to be \[
	h_f=(u_f)^*,
	\]
	which is the Legendre transformation of the base function for $f.$
	Moreover, we propose the  definition of s-Aleksandrov function here.
	\begin{definition}
		Let $u:\R^{n}\to[0,\infty]$ be a lower semi-continuous
		function (which may or may not be convex)  with $u(x) \geq u(o) =0$ for all $x \in \R^n$. For $s\in(-\infty,\infty)$, the s-Alexandrov Function
		of $u$ is $A[u]_s=(1-su^{\ast})_+^{1/s}.$ 
	\end{definition}	
	Note that $f$ is the largest  $s$-concave
	function with $h_{f}\le(u_f)^*$.
	We then define the $L_{p,s}$ Asplund summation for $0<p<1$ using the Legendre transformation of the base functions, i.e., $h_f=(u_f)^*$ as follows.
	\begin{definition}
		For $0<p<1$, $s\in (-\infty,\infty)$, given $f,g \in \mathcal{F}_s(\R^n)$ with $h_{f},h_{g}\ge0$, we define
		the $L_{p,s}$ Asplund summation $\alpha\cdot_{p,s}f\star_{p,s}\beta\cdot_{p,s} g$ with weights $\alpha,\beta\geq0$ as
		\[
		\alpha\cdot_{p,s}f\star_{p,s}\beta\cdot_{p,s} g=A\left[\left(\alpha h_{f}^{p}+\beta h_{g}^{p}\right)^{1/p}\right]_s.
		\]
	\end{definition}
	The above definition recovers the results of  Asplund summation for log-concave functions in \cite{Rotem2} if $s=0$.	
	
	One of our main results is that we show  the $L_p$ supremal-convolution and the $L_{p,s}$ Asplund summation using the base functions coincide with each other in $\mathcal{F}_0(\R^n)$ which generalized the results of  \cite[Proposition~10]{Rotem} in next subsection. This connection is however only works for the coefficients $1-t$ and $t$, while for more general coefficients it needs more homogeneity restrictions. See more references in \cite{FXY, Rotem, Rotem2}.
	\subsection{Relation between $L_{p,s}$ convolutions and $L_{p,s}$ Asplund summation}\label{subsection23}

	Next inspired by (\ref{compare}) for $p=1$, we compare the $L_{p,s}$ convolutions and $L_{p,s}$ Asplund summations for functions defined above for different cases of $p$ and $s.$  Together with the fact that  established by Artstein-Avidan and Milman  \cite{AM1,AM2} that the Legendre transformation is the only duality on the class $\text{Cvx}(\R^n)$, that is, the only transformation that satisfies the conditions:
	\[
	u^{**} = u \quad \text{ and } \quad  u^* \geq v^* \text{ whenever } u,v \in \text{Cvx}(\R^n) \text{ satisfy } u \leq v,
	\]
	we give a detailed proof of the following properties.
	\begin{proposition}\label{t:functionalpclosure} 
		(1) Let $p\geq 1$, $f,g \in \mathcal{F}_0(\R^n)$ ($s=0$) be of the form $f=e^{-u}$ and $g = e^{-v}$ 
		for some $u,v \in C_s(\R^n)$. Then the following equality holds: 
		\[
		[(1-t) \times_{p,0} f] \oplus_{p,0} [t \times_{p,0} g](z) = e^{-\big(M_p^{((1-t),t)}(u^*(z),v^*(z))\big)^*}.
		\]
		
		(2) Let $p\geq1$, $s \neq 0$, $f,g \in \mathcal{F}_s(\R^n)$ be of the form $f=(1-su)_+^{1/s}$ and $g = (1-sv)_+^{1/s}$, 
		for some $u,v \in C_s(\R^n)$. Then the following inequality holds: 
		\[
		[(1-t) \times_{p,s} f] \oplus_{p,s} [t \times_{p,s} g](z) \leq \big((1-s(M_p^{((1-t),t)}(u^*(z),v^*(z)))^*\big)_+^{1/s}.
		\]
		
		(3) Let $0<p<1$, $f,g \in \mathcal{F}_s(\R^n)$ be of the form $f=(1-su)_+^{1/s}$ and $g = (1-sv)_+^{1/s}$ 
		for some $u,v \in C_s(\R^n)$. Then the following inequality holds: 
		\[
		[(1-t) \times_{p,s} f] \oplus_{p,s} [t \times_{p,s} g](z) \geq \big(1-s(M_p^{((1-t),t)}(u^*(z),v^*(z)))^*\big)_+^{1/s}\quad \text{for}\quad s>0;
		\]
		\[
		[(1-t) \times_{p,s} f] \oplus_{p,s} [t \times_{p,s} g](z) \leq \big(1-s(M_p^{((1-t),t)}(u^*(z),v^*(z)))^*\big)_+^{1/s}\quad \text{for}\quad s<0.
		\]
	\end{proposition}

	\begin{proof} (1) For  $p\geq1$,  it follows from the definition of $L_{p,s}$ supremal-convolution for $s=0$ in Definition \ref{t:wholepoperations} and (\ref{infcon}) that
		\begin{align*}
		([(1-t) \times_{p,0} f] \oplus_{p,0} [t \times_{p,0} g])(z) 
		&= \sup_{0 \leq \lambda \leq 1}\left[(C_{p,\lambda,t} \times_0 f) \oplus_0 (D_{p,\lambda,t}\times_0 g)(z) \right]\\
		&= \sup_{0 \leq \lambda \leq 1} \left[ \sup_{z = C_{p,\lambda,t} x + D_{p,\lambda,t}y} e^{-C_{p,\lambda,t} u(x) + D_{p,\lambda,t} v(y)}\right]\\
		&= \sup_{0 \leq \lambda \leq 1} e^{-\inf_{z = C_{p,\lambda,t} x + D_{p,\lambda,t}y}[C_{p,\lambda,t} u(x)+D_{p,\lambda,t}  v(y)]}\\
		&= \sup_{0 \leq \lambda \leq 1} e^{-[C_{p,\lambda,t}\times u\square D_{p,\lambda,t} \times v](z)}
		\\
		&= e^{- \inf_{0\leq\lambda\leq 1} [C_{p,\lambda,t} u^*(z) + D_{p,\lambda,t} v^*(z)]^* }\\
		&= e^{- \left[\sup_{0 \leq \lambda \leq 1}(C_{p,\lambda,t} u^*(z) + D_{p,\lambda,t} v^*(z))\right]^*}\\
		&= e^{-(M_p^{((1-t),t)}(u^*(z),v^*(z)))^*}.
		\end{align*}
		Above we have also used the identity \eqref{e:LegendreInfconvolution}, \cite[Theorem~11.23 (d)]{Rock2} together with the Fenchel-Moreau theorem (i.e. the Legendre transform is an involution on proper lower semi-continuous convex functions), and Lemma \ref{supremallp} (1).	
		
		(2) For $p\geq1$ and  $s>0$, we observe similarly that
		\begin{align*}
		&  ([(1-t) \times_{p,s} f] \oplus_{p,s} [t \times_{p,s} g])(z)^s \\&= \sup_{0 \leq \lambda \leq 1}\left[(C_{p,\lambda,t} \times_s f) \oplus_s (D_{p,\lambda,t}\times_s g)(z) \right]\\	
		&=\sup_{0 \leq \lambda \leq 1} \sup_{z = C_{p,\lambda,t} x + D_{p,\lambda,t}y}[ C_{p,\lambda,t}-s C_{p,\lambda,t} u(x) +D_{p,\lambda,t}-s D_{p,\lambda,t} v(y)] \\
		&= \sup_{0 \leq \lambda \leq 1} \big[ C_{p,\lambda,t}+D_{p,\lambda,t}-s\inf_{z = C_{p,\lambda,t} x + D_{p,\lambda,t}y}(C_{p,\lambda,t} u(x)+ D_{p,\lambda,t} v(y))\big]\\
		&= \sup_{0 \leq \lambda \leq 1} \big[ C_{p,\lambda,t}+D_{p,\lambda,t}-s(C_{p,\lambda,t} u^*(z) + D_{p,\lambda,t} v^*(z))^*\big] \\
		&\leq \sup_{0 \leq \lambda \leq 1} ( C_{p,\lambda,t}+D_{p,\lambda,t})-s\inf_{0 \leq \lambda \leq 1}(C_{p,\lambda,t} u^*(z) + D_{p,\lambda,t} v^*(z))^* \\
		&\leq 1 -s(M_p^{((1-t),t)}(u^*(z),v^*(z)))^*.
		\end{align*}
		Above we have used the H\"older's identity $C_{p,\lambda,t}+D_{p,\lambda,t}\leq 1$, \eqref{e:LegendreInfconvolution}, \cite[Theorem~11.23 (d)]{Rock2} together with the Fenchel-Moreau theorem, and Lemma \ref{supremallp} (1).	
		
		For $s<0$, it can be proved in a similar way as		
		\begin{align*}
		&\!\!\!\!\!\!([(1-t) \times_{p,s} f] \oplus_{p,s} [t \times_{p,s} g])(z)^s \\&= \sup_{0 \leq \lambda \leq 1}\left[(C_{p,\lambda,t} \times_s f) \oplus_s (D_{p,\lambda,t}\times_s g)(z) \right]\\	
		&=\sup_{0 \leq \lambda \leq 1} \sup_{z = C_{p,\lambda,t} x + D_{p,\lambda,t}y}[ C_{p,\lambda,t}-s C_{p,\lambda,t} u(x) +D_{p,\lambda,t}-s D_{p,\lambda,t} v(y)] \\
		&= \sup_{0 \leq \lambda \leq 1} \big[ C_{p,\lambda,t}+D_{p,\lambda,t}-s\sup_{z = C_{p,\lambda,t} x + D_{p,\lambda,t}y}(C_{p,\lambda,t} u(x)+ D_{p,\lambda,t} v(y))\big] \\
		&\geq \sup_{0 \leq \lambda \leq 1} \big[ C_{p,\lambda,t}+D_{p,\lambda,t}-s\inf_{z = C_{p,\lambda,t} x + D_{p,\lambda,t}y}(C_{p,\lambda,t} u(x)+ D_{p,\lambda,t} v(y))\big] \\
		&= \sup_{0 \leq \lambda \leq 1} \big[ C_{p,\lambda,t}+D_{p,\lambda,t}-s(C_{p,\lambda,t} u^*(z) + D_{p,\lambda,t} v^*(z))^*\big] \\
		&\geq \sup_{0 \leq \lambda \leq 1} ( C_{p,\lambda,t}+D_{p,\lambda,t})-s\inf_{0 \leq \lambda \leq 1}(C_{p,\lambda,t} u^*(z) + D_{p,\lambda,t} v^*(z))^* \\
		&\geq 1 -s(M_p^{((1-t),t)}(u^*(z),v^*(z)))^*,
		\end{align*}
		as desired.
		
		(3) For $0<p<1$ and  $s> 0$, we compute
		\begin{align*}
		&\!\!\!\!\!\!([(1-t) \times_{p,s} f] \oplus_{p,s} [t \times_{p,s} g])(z)^s \\&= \inf_{0\leq\lambda\leq1}\left[(C_{p,\lambda,t} \times_s f) \oplus_s (D_{p,\lambda,t}\times_s g)(z) \right]\\
		&=\inf_{0\leq\lambda\leq 1} \sup_{z = C_{p,\lambda,t} x + D_{p,\lambda,t}y}[ C_{p,\lambda,t}-s C_{p,\lambda,t} u(x) +D_{p,\lambda,t}-s D_{p,\lambda,t} v(y)] \\
		&= \inf_{0\leq\lambda\leq1} \big[ C_{p,\lambda,t}+D_{p,\lambda,t}-s\inf_{z = C_{p,\lambda,t} x + D_{p,\lambda,t}y}(C_{p,\lambda,t} u(x)+ D_{p,\lambda,t} v(y))\big] \\
		&= \inf_{0\leq\lambda\leq1} \big[ C_{p,\lambda,t}+D_{p,\lambda,t}-s(C_{p,\lambda,t} u^*(z) + D_{p,\lambda,t} v^*(z))^*\big] \\
		&\geq \inf_{0\leq\lambda\leq1} ( C_{p,\lambda,t}+D_{p,\lambda,t})-s\sup_{0\leq \lambda\leq 1}(C_{p,\lambda,t} u^*(z) + D_{p,\lambda,t} v^*(z))^* \\
		&\geq 1 -s(M_p^{((1-t),t)}(u^*(z),v^*(z)))^*.
		\end{align*}
		Above we have used the reverse H\"older's identity $C_{p,\lambda,t}+D_{p,\lambda,t}\geq 1$ for $0<p<1$, \eqref{e:LegendreInfconvolution}, \cite[Theorem~11.23 (d)]{Rock2}, and Lemma \ref{supremallp} (3).	
		
		For $s<0$, it can be proved in a similar way as		
		\begin{align*}
		&\!\!\!\!\!\![(1-t) \times_{p,s} f] \oplus_{p,s} [t \times_{p,s} g](z)\\	
		&=\inf_{0 \leq \lambda \leq 1} \sup_{z = C_{p,\lambda,t} x + D_{p,\lambda,t}y}\big[[ C_{p,\lambda,t}-s C_{p,\lambda,t} u(x) +D_{p,\lambda,t}-s D_{p,\lambda,t} v(y)] \big]^{1/s}\\
		&=\big[\sup_{0 \leq \lambda \leq 1} \inf_{z = C_{p,\lambda,t} x + D_{p,\lambda,t}y}[ C_{p,\lambda,t}-s C_{p,\lambda,t} u(x) +D_{p,\lambda,t}-s D_{p,\lambda,t} v(y)] \big]^{1/s}\\
		&=\big[\sup_{0 \leq \lambda \leq 1}\big([ C_{p,\lambda,t}+D_{p,\lambda,t}]-s  \inf_{z = C_{p,\lambda,t} x + D_{p,\lambda,t}y}[C_{p,\lambda,t} u(x) + D_{p,\lambda,t} v(y)] \big)\big]^{1/s}\\
		&\leq \big[1-s\sup_{0 \leq \lambda \leq 1} (C_{p,\lambda,t} u^*(z) + D_{p,\lambda,t} v^*(z))^* \big]^{1/s}\\
		&= \big[1-s(\inf_{0 \leq \lambda \leq 1}  C_{p,\lambda,t} u^*(z) + D_{p,\lambda,t} v^*(z))^* \big]^{1/s}\\
		&\leq \big[1 -s(M_p^{((1-t),t)}(u^*(z),v^*(z)))^*\big]^{1/s},
		\end{align*}
		as desired.
	\end{proof}
	
	\section{New $L_p$-Borell-Brascamp-Lieb type inequalities }\label{section3}
	In this section, we will present several  $L_p$-Borell-Brascamp-Lieb type inequalities related to the $L_{p,s}$ supremal-convolution. Firstly, in Subsection \ref{subsection31} we extend the  Borell-Brascamp-Lieb inequality in \cite{BCF}[Theorem 4.1 and 4.2] to the $L_p$ case  as Theorem \ref{t:LpBorell-Brascamp-Lieb functionals}  and Theorem \ref{t:1dgeneralBorell-Brascamp-Lieb }. Secondly, we give different improvement methods of $L_p$-Borell-Brascamp-Lieb inequality in \cite{RX} for $s\geq0$ to $s\in[-\infty,\infty]$ including using mass transportation with matrix inequality and applying the result of classical Borell-Brascamp-Lieb inequality in Subsection \ref{subsection32}.

	\subsection{A Novel $L_p$-Borell-Brascamp-Lieb inequality for  $p\geq1$}\label{subsection31}
	Recall in \cite[Page 22]{BCF} that functional  $\Omega: \mathscr{B
	}\rightarrow\R_+$ (for example, capacity or the  measure of a set in $\R^n$), where $\mathscr{B}$ denotes the class of Borel subsets of $\R^n$, is said to be monotone if,
	\[
	\Omega(A_0) \leq \Omega(A_1), \quad \text{ whenever } A_0 \subset A_1,
	\]
	and $\gamma$-concave, with $\gamma \in [-\infty,\infty]$ and $t \in [0,1]$, if 
	\[
	\Omega((1-t) A_0 + t A_1)) \geq M_{\gamma}^{((1-t),t)}(\Omega(A_0),\Omega(A_1))
	\]
	for all Borel sets $A_0,A_1\in\mathscr{B}$, with $\Omega(A_0), \Omega(A_1) > 0$. We always take the convention that $\Omega(\emptyset) = 0$ and $\Omega(A) > 0$ implies that $A$ is non-empty. For example, given a compact set $S$ in the $n$-dimensional Euclidean space $\mathbb R^n$ ($n\geq2$), the variational $p$-capacity of $S$ for $p\in (1,n)$ is defined by
	$$
	\hbox{Cap}_p(S)=\inf\left\{\int_{\mathbb R^n}|\nabla f|^p\,dx:\ f\in C_c^\infty(\mathbb R^n) \ \text{and}\ \ f(x)\ge 1\   \text{for all} \ x\in S\right\},
	$$
	where $C^\infty_c(\mathbb R^n)$ denotes the class of all infinitely differentiable functions with compact support in $\mathbb R^n$. If $\Omega=\hbox{Cap}_p$, then $\hbox{Cap}_p$ is $\frac{1}{n-p}$ concave for $p\in (1,n)$ (as shown in \cite[Theorem 1]{CS}) on the class $\mathcal{K}_o^n$ of convex bodies in $\R^n$.

	This definition can be extended to a non-negative Borel measurable functions $f$ with its super level sets by setting 
	\[
	\widetilde{\Omega}(f):= \int_0^{\infty} \Omega(\{x\in\R^n \colon f(x) \geq r\}) dr. 
	\]
	Inspired by \cite{BCF}, we change the  original condition   with extra the power condition with parameter $\ga$ for the triple of functions $(h,f,g)$, particularly for terms inside
	$h$ (in dimension $1$) in Theorem \ref{t:1dgeneralBorell-Brascamp-Lieb } in the $L_p$ case and obtain the so-called  $L_{p,\ga}$ Borell-Brascamp-Lieb inequality in $\R$ generalizing \cite[Theorem 4.1]{BCF}. Further we establish a $L_p$-Borell-Brascamp-Lieb  type inequality in Theorem \ref{t:LpBorell-Brascamp-Lieb functionals} with the $\widetilde{\Omega}$ which recovers a slight modification of Borell-Brascamp-Lieb  inequality \cite[Theorem 4.2]{BCF} for $p=1$ as follows.

	\begin{theorem}\label{t:LpBorell-Brascamp-Lieb functionals} Let $\Omega:\mathscr{B}\rightarrow\R_+$ be $\alpha$-concave. Let $p,q \in[1,\infty]$ be such that  $1/p+1/q=1$. Let $\alpha \in [-1,+\infty]$ and $\gamma \in [-\alpha, \infty)$. Suppose that $h,f,g \colon \R^n \to \R_+$ are a triple of integrable Borel measurable (respectively, quasi-concave functions) that satisfy the condition 
		\begin{equation}\label{e:mainassumption12}
		\begin{split}
		h\left(C_{p,\lambda,t}x + D_{p,\lambda,t}y  \right) \geq\left[C_{p,\lambda,t}f(x)^{\alpha} + D_{p,\lambda,t} g(y)^{\alpha} \right]^{\frac{1}{\alpha}}
		\end{split}
		\end{equation}
		for every $x \in \text{supp}(f)$, $y \in \text{supp}(g)$, and $\lambda \in (0,1)$ whenever $f(x)g(y) >0$. 
		Then the following inequality holds:
		\begin{equation*}\label{e:functionalBorell-Brascamp-Lieb conclusion}
		\begin{split}
		\widetilde{\Omega}(h) \geq \left[(1-t) \widetilde{\Omega}(f)^{\be} + t \widetilde{\Omega}(g) ^{\be}\right]^{\frac{1}{\beta}}, \quad \be = \frac{p\alpha\gamma}{\alpha+\gamma}.
		\end{split}
		\end{equation*}
		
	\end{theorem}
	
	\begin{remark}
		The version of Theorem \ref{t:LpBorell-Brascamp-Lieb functionals} for $p=1$ originally appeared in \cite[Theorem 4.2]{BCF}, and  the set inclusion (4.13) as stated has to be modified slightly to follow  \cite[Theorem 4.1]{BCF} applied the assumptions of \cite[Theorem 4.2]{BCF}. Therefore, we include full proof of details in the $L_p$ case inspired by the proof in  \cite{BCF}.
	\end{remark}

	The proof of the Theorem relies on the following one-dimensional result when $$\widetilde{\Omega}(h)=\int_{\R}h(x)dx.$$
	
	\begin{theorem}{\textbf{\emph{($L_{p,\ga}$ Borell-Brascamp-Lieb inequality in $\R$)}}}\label{t:1dgeneralBorell-Brascamp-Lieb }
		Let $p,q \in[1,\infty]$ be such that  $1/p+1/q=1$. Let $t \in (0,1)$, $\alpha \in [-1,+\infty]$ and $\gamma \in [-\alpha, \infty)$.  Let $h,f,g \colon (0,\infty) \to \R_+$ be a triple of integrable functions that satisfy the condition 
		\begin{equation}\label{e:assumption1}
		\begin{split}
		h\left(\left(C_{p,\lambda,t}x^{\ga} + D_{p,\lambda,t}y^{\ga} \right)^{\frac{1}{\ga}} \right) \geq\left[C_{p,\lambda,t}f(x)^{\alpha} + D_{p,\lambda,t} g(y)^{\alpha} \right]^{\frac{1}{\alpha}}
		\end{split}
		\end{equation}
		for every $x \in \text{supp}(f)$, $y\in \text{supp}(g)$, and $\lambda \in (0,1)$ whenever $f(x)g(y)>0$. Then the following integral inequality holds:
		\begin{equation}\label{e:conclusion}
		\int_0^{\infty}h(x) dx \geq \left((1-t) \left(\int_0^{\infty} f(x) dx\right)^{\be} +t \left(\int_0^{\infty} g(x) dx \right)^{\be} \right)^{\frac{1}{\be}},
		\end{equation}
		where $\beta = \frac{p\al \ga}{\al + \ga}$.
	\end{theorem}

	The proof of Theorem~\ref{t:1dgeneralBorell-Brascamp-Lieb } is postponed until the next section, as it requires division into several steps. For now we prove Theorem~\ref{t:LpBorell-Brascamp-Lieb functionals} firstly.
	
	\begin{proof}[Proof of Theorem~\ref{t:LpBorell-Brascamp-Lieb functionals}] We denote by 
		\[
		C_m(r) = \{x \in \R^n \colon m(x) \geq r \}
		\]
		the super-level set for any Borel measurable function $m$. By the hypothesis \eqref{e:mainassumption12} placed on the triple of functions $h,f,g$, one has \begin{equation}\label{e:levelsetinclusion}
		C_h(\tau_{\lambda}^{\gamma}) \supset C_{p,\lambda,t} C_f(r) + D_{p,\lambda,t} C_g(s), \quad \tau_{\lambda}^{a}:=[C_{p,\lambda,t}r^{a}+D_{p,\lambda,t}s^{a}]^{\frac{1}{a}}, \ a\in[-\infty,\infty]
		\end{equation}
		holds for all $\lambda \in (0,1)$ whenever $r,s>0$ satisfies $\Omega(C_f(r))>0$ and $\Omega(C_g(r))>0$. 	Indeed, if for some fixed $\lambda_0 \in (0,1)$, $z \in C_{p,\lambda_0,t} C_f(r) + D_{p,\lambda_0,t} C_g(s)$, then there exist some $x \in C_f(r)$ and $y\in C_g(s)$ such that 
		\[
		z = C_{p,\lambda_0,t}x+D_{p,\lambda_0,t}y.
		\]
		Using  the assumption \eqref{e:mainassumption12}, we have that 
		\begin{align*}
		h(z) &= h\left(C_{p,\lambda_0,t}x+D_{p,\lambda_0,t}y\right)\\
		&\geq \left[C_{p,\lambda_0,t}f(x)^{\al}+D_{p,\lambda_0,t}g(y)^{\al} \right]^{\frac{1}{\alpha}}\\
		& \geq \tau_{\lambda_0}^{\alpha},
		\end{align*}
		which establishes the inclusion \eqref{e:levelsetinclusion} for every fixed $\lambda_0 \in (0,1)$. 
		
		Consider the functions $\bar{h},\bar{f},\bar{g} \colon (0,\infty) \to \R_+$, the composition of $\Omega$ and super level sets, defined, respectively, by 
		\[
		\bar{h}(r) := \Omega(C_h(r)), \quad \bar{f}(r) := \Omega(C_f(r)), \quad \bar{g}(r) := \Omega(C_g(r)).
		\]
		By the monotonicity and $\alpha$-concavity of $\Omega$ and the inclusion \eqref{e:levelsetinclusion}, this triple of functions satisfy 
		\begin{align*}
		&\bar{h}\left[\left(C_{p,\lambda,t}r^{\ga} + D_{p,\lambda,t}r^{\ga} \right)^{\frac{1}{\ga}} \right] \\
		&= \Omega\left(C_h(\tau_{\lambda}^{\gamma}) \right)\\
		&\geq \Omega\left( C_{p,\lambda,t} C_f(r) + D_{p,\lambda,t}C_g(s)\right)\\
		&\geq \left[C_{p,\lambda,t}\Omega(C_f(r))^{\alpha} +\left(1-C_{p,\lambda,t}\right) \Omega\left(\frac{D_{p,\lambda,t}}{1-C_{p,\lambda,t}} C_g(s)\right)^{\alpha}\right]^{\frac{1}{\alpha}}\\
		&\geq \left[C_{p,\lambda,t} \bar{f}(r)^{\alpha}+ D_{p,\lambda,t}\bar{g}(s)^{\alpha} \right]^{\frac{1}{\alpha}}
		\end{align*}
		holds for every $r,s>0$  and $\lambda \in (0,1)$ whenever $\bar{f}(r)\bar{g}(s) >0$. Above we have used H\"older's inequality to conclude that for $p\geq1$,
		$
		C_{p,\lambda,t}+D_{p,\lambda,t} \leq 1. 
		$
		
		Therefore, the triple of functions $\{\bar{f},\bar{g},\bar{h}\}$ satisfy the hypothesis of Theorem~\ref{t:1dgeneralBorell-Brascamp-Lieb } (\ref{e:assumption1}), and therefore
		\begin{align*}
		\widetilde{\Omega}(h) &= \int_0^{\infty}\bar{h}(r) dr\\
		&\geq \left[(1-t) \left(\int_0^{\infty}\bar{f}(r) dr\right)^{\beta} + t \left(\int_0^{\infty}\bar{g}(r)dr\right)^{\beta} \right]^{\frac{1}{\beta}}\\
		&= \left[(1-t) \Omega\left(f \right)^{\beta} + t \Omega\left(g\right)^{\beta} \right]^{\frac{1}{\beta}},
		\end{align*}
		where $\beta = \frac{p\alpha\gamma}{\alpha+\gamma}$, as desired formula (\ref{e:conclusion}). 
	\end{proof}
	
	The proof of Theorem~\ref{t:1dgeneralBorell-Brascamp-Lieb } is inspired by the work of Ball \cite{Ball} and Bobkov, Colesanti, and Fragal{\`a} \cite{BCF} for different cases for $\ga$ with details as follows.
	\begin{proof}[\textbf{Proof of Theorem~\ref{t:1dgeneralBorell-Brascamp-Lieb }}]
		\textbf{\circled{1} The case $\gamma =1$.} Assume that $\gamma =1$ and let  $\alpha \in [-1, \infty]$. For the case $\alpha \geq 0$, we already handled in our other paper \cite{RX}.  Therefore, we may assume that $\alpha \in [-1,0).$
		
		Fix $\lambda \in (0,1)$.  As all functions involved are integrable, we may assume, without loss of generality, that $f$ and $g$ are bounded with non-zero maximums. Set 
		\[
		M_{\lambda}=\left[C_{p,\lambda,t}\|f\|_{\infty}^{\alpha} + D_{p,\lambda,t} \|g\|_{\infty}^{\alpha}\right]^{\frac{1}{\alpha}}.
		\]
		Using the assumptions placed on the triple $h,f,g$  (\ref{e:assumption1}), we see that, for any $x \in \text{supp}(f)$ and $y \in \text{supp}(g)$, one has 
		\begin{align*}
		&h\left(C_{p,\lambda,t}x + D_{p,\lambda,t}y\right)\\
		&\geq \left[C_{p,\lambda,t}f(x)^{\alpha}+D_{p,\lambda,t}g(y)^{\alpha} \right]^{1/\alpha}\\
		&=M_{\lambda}\left[C_{p,\lambda,t}\left(\frac{\|f\|_{\infty}}{M_{\lambda}} \right)^{\alpha}\Bar{f}(x)^{\alpha}+D_{p,\lambda,t}\left(\frac{\|g\|_{\infty}}{M_{\lambda}}\right)^{\alpha}\bar{g}(y)^{\alpha} \right]^{1/\alpha}\\
		&=M_{\lambda}\left[(1-\theta)\bar{f}(x)^{\alpha} + \theta \bar{g}(y)^{\alpha}\right]^{\frac{1}{\alpha}}, \quad  \theta= D_{p,\lambda,t}\left(\frac{\|g\|_{\infty}}{M_{\lambda}}\right)^{\alpha}, \bar{f}= \frac{f}{\|f\|_{\infty}}, \bar{g}= \frac{g}{\|g\|_{\infty}},\\
		&\geq M_{\lambda}\min\{\bar{f}(x), \bar{g}(y) \}. 
		\end{align*}
		
		Therefore, by letting $h_{\lambda}:= \frac{h}{M_{\lambda}}$, we see that 
		\[
		\{h_{\lambda} \geq \eta\} \supset C_{p,\lambda,t}\{\bar{f} \geq \eta\} + D_{p,\lambda,t} \{\bar{g}\geq \eta\}
		\]
		for all $\eta \in [0,1]$ whenever $x \in \{\bar{f} \geq \eta\}$ and $y \in \{\bar{g}\geq \eta\}.$ Hence, using Fubini's theorem and the Brunn-Minkowski inequality in dimension one $\V_1(A+B)\geq \V_1(A)+\V_1(B)$ where $\V_1$ denotes the volume of set in $\R$, we see that 
		\begin{align*}
		\int_0^{\infty}h(x) dx 
		&= M_{\lambda}\int_0^{\infty}h_{\lambda}(x) dx
		\\
		&= M_{\lambda} \int_0^1\V_1(\{h_{\lambda} \geq \eta \} )d \eta\\
		&= M_{\lambda} \int_0^1 \V_1(C_{p,\lambda,t}\{\bar{f} \geq \eta\} + D_{p,\lambda,t} \{\bar{g} \geq \eta\})d \eta\\
		&= M_{\lambda} \left(C_{p,\lambda,t} \left(\int_0^{\infty} \bar{f}(x) dx \right) + D_{p,\lambda,t} \left(\int_0^{\infty} \bar{g}(x) dx\right)\right)\\
		&=\left[(1-\lambda) \left(\left(\frac{1-t}{1-\lambda}\right)^{\frac{1}{\alpha p}}\|f\|_{\infty}\right)^{\alpha}+\lambda \left(\left(\frac{t}{\lambda}\right)^{\frac{1}{\alpha p}}\|g\|_{\infty}\right)^{\alpha}\right]^{\frac{1}{\alpha}} \times\\&
		\left[(1-\lambda) \left(\frac{1-t}{1-\lambda}\right)^{\frac{1}{p}}\left(\int_0^{\infty} \bar{f}(x) dx \right) + \lambda\left(\frac{t}{\lambda}\right)^{\frac{1}{p}} \left(\int_0^{\infty} \bar{g}(x) dx \right)\right]
		\\&\geq\left[C_{p,\lambda,t} \left( \int_0^{\infty} f(x) dx \right)^{\frac{\alpha}{\alpha+1}} + D_{p,\lambda,t} \left(\int_0^{\infty} g(x) dx \right)^{\frac{\alpha}{\alpha+1}}\right]^{\frac{\alpha+1}{\alpha}},
		\end{align*}
		where in the last line we have used the fact that $\alpha> -1$ together with the generalized H\"older inequality; i.e., for all $u_1,u_2,v_1,v_2 \geq 0$ and
		$\lambda \in (0,1)$, $t\in[0,1]$, it holds
		\begin{equation}\label{holder}
		M_{\alpha_1}^{C_{p,\lambda,t}, D_{p,\lambda,t}}(u_1,v_1) M_{\alpha_2}^{C_{p,\lambda,t}, D_{p,\lambda,t}}(u_2,v_2)\,
		\geq\, M_{\alpha_0}^{C_{p,\lambda,t}, D_{p,\lambda,t}}(u_1 u_2,v_1 v_2),
		\end{equation}
		whenever
		\[
		\alpha_1 + \alpha_2 > 0, \qquad
		\frac{1}{\alpha_0} = \frac{1}{\alpha_1} + \frac{1}{\alpha_2}.
		\]

		Therefore, as $\lambda$ was arbitrarily fixed, we actually proved that 
		\[
		\int_0^{\infty}h(x) dx \geq \sup_{0 < \lambda < 1} \left[C_{p,\lambda,t} \left( \int_0^{\infty} f(x) dx \right)^{\frac{\alpha}{\alpha+1}} + D_{p,\lambda,t} \left(\int_0^{\infty} g(x) dx \right)^{\frac{\alpha}{\alpha+1}}\right]^{\frac{\alpha+1}{\alpha}}. 
		\]
		
		By optimizing over $\lambda$, with $\alpha \in (-1,0)$, together with Lemma \ref{supremallp} (2), we see 
		\[
		\int_0^{\infty}h(x) dx \geq \left[(1-t) \left( \int_0^{\infty} f(x) dx \right)^{\frac{p\alpha}{\alpha+1}} + t\left(\int_0^{\infty} g(x) dx \right)^{\frac{p\alpha}{\alpha+1}}\right]^{\frac{\alpha+1}{\alpha p}},
		\]
		which completes the proof for $\gamma  = 1.$

		\textbf{\circled{2}  The case $\gamma= 0$.}	Suppose that $\gamma = 0$.  Consider the functions $m,d,n: \R \setminus\{0\} \to \R_+$ defined by 
		\[
		m(x):= h(e^x)e^x, \quad d(x):=f(e^x)e^x, \quad n(x) := g(e^x)e^x. 
		\]
		Then, for any $e^x \in \text{supp}(f)$, $e^y \in \text{supp}(g)$, and $\lambda \in (0,1)$, applying the assumption \eqref{e:assumption1}, one has 
		\begin{equation}\label{e:lc1}
		\begin{split}
		m\left(C_{p,\lambda,t}x + D_{p,\lambda,t}y\right) &= h\left( e^{C_{p,\lambda,t}x + D_{p,\lambda,t}y}\right)e^{C_{p,\lambda,t}x + D_{p,\lambda,t}y}\\
		&\geq \left[ f(e^x)e^x \right]^{C_{p,\lambda,t}}\left[ g(e^y)e^y \right]^{D_{p,\lambda,t}}\\
		&=d(x)^{C_{p,\lambda,t}}n(y)^{D_{p,\lambda,t}}.
		\end{split}
		\end{equation}
		
		Recall that the $L_p$-Pr\'ekopa-Leindler inequality for product measures \cite{RX} with quasi-concave densities states that 
		let $f,g,h \colon \R^n\to \R_+$ be a triple of measurable functions, with $f,g$ weakly unconditional and positively decreasing, that satisfy the condition
		\begin{equation}\label{e:lpprekopaleindlerassumption0}
		h(C_{p,\lambda,t}x + D_{p,\lambda,t}y) \geq f(x)^{C_{p,\lambda,t}}g(y)^{D_{p,\lambda,t}}
		\end{equation}
		for every $x \in \text{supp}(f), y \in \text{supp}(g)$, and every $0 < \lambda < 1$. The the following integral inequality holds:
		\begin{equation*}\label{e:lpprekopaleindlerconclusion0}
		\int_{\R^n} h d\mu \geq \sup_{0 < \lambda < 1}\left\{\left[\left(\frac{1-t}{1-\lambda}\right)^{1-\lambda} \left(\frac{t}{\lambda}\right)^{\lambda}\right]^{\frac{n}{p}}\left(\int_{\R^n} f^{\left(\frac{1-t}{1-\lambda}\right)^{\frac{1}{p}}} d\mu \right)^{1-\lambda}\left(\int_{\R^n}g^{\left(\frac{t}{\lambda}\right)^{\frac{1}{p}}} d\mu \right)^{\lambda} \right\}.
		\end{equation*}
		
		According to inequality \eqref{e:lc1}, the triple of functions $(m,d,n)$ satisfy the condition in dimension 1 (\ref{e:lpprekopaleindlerassumption0}), and therefore
		\begin{equation}
		\int_{\R} m dx \geq \sup_{0 < \lambda < 1}\left\{\left[\left(\frac{1-t}{1-\lambda}\right)^{1-\lambda} \left(\frac{t}{\lambda}\right)^{\lambda}\right]^{\frac{n}{p}}\left(\int_{\R} d^{\left(\frac{1-t}{1-\lambda}\right)^{\frac{1}{p}}} dx \right)^{1-\lambda}\left(\int_{\R}n^{\left(\frac{t}{\lambda}\right)^{\frac{1}{p}}} dx \right)^{\lambda} \right\}.
		\end{equation}
		Therefore by choosing $\lambda=t$, we can see that
		\[
		\int_{\R} m(x) dx \geq \left(\int_{\R} d(x) dx \right)^{1-t} \left(\int_{\R} n(x)dx \right)^t.
		\]
		Finally, note that 
		\[
		\int_{\R}m(x) dx = \int_{\R}h(e^x)e^x dx = \int_0^{\infty}h(x) dx,
		\]
		and the same with the pairs $(d,f)$ and $(n,g)$.  This completes the proof of the theorem in the case $\gamma = 0$.

		Next we consider $\ga \neq 0,1$.	Suppose that $\gamma \in (-\infty,1) \setminus \{0\}.$ Let $-\gamma\leq\alpha\leq\infty$ with $\gamma>-\infty.$ Consider the triple of functions $w,u,v$ defined by 
		\[
		w(x) = h(x^{1/\gamma}), \quad u(x) = f(x^{1/\gamma}), \quad v(x) = g(x^{1/\gamma}). 
		\]	
		Using the assumption \eqref{e:assumption1}, we see that 
		\begin{equation}\label{e:one}
		\begin{split}
		w\left(C_{p,\lambda,t} x + D_{p,\lambda,t}y \right) &= h\left(\left(C_{p,\lambda,t} x + D_{p,\lambda,t}y \right)^{\frac{1}{\ga}}\right)\\
		&\geq \left[C_{p,\lambda,t} f(x^{1/\gamma})^{\al} + D_{p,\lambda,t}g(y^{1/\ga})^{\al} \right]^{\frac{1}{\al}}\\
		&=\left[C_{p,\lambda,t}u(x)^{\al} + D_{p,\lambda,t}v(y)^{\al}  \right]^{\frac{1}{\al}}
		\end{split}
		\end{equation}
		holds whenever $x^{1/\ga} \in \text{supp}(f)$, $y^{1/\ga}\in \text{supp}(g)$, and any $\lambda \in (0,1).$
		
		Set $\delta = \frac{\gamma}{1-\gamma}$, and fix $\lambda \in (0,1)$.  Let 
		\[
		A = \left[C_{p,\lambda,t}+D_{p,\lambda,t} \right]^{\frac{1}{\al}}, B =\left[C_{p,\lambda,t}+D_{p,\lambda,t} \right]^{\frac{1}{\delta}},
		\]
		and 
		\[
		\theta = \frac{D_{p,\lambda,t}}{C_{p,\lambda,t}+D_{p,\lambda,t}} \in [0,1]. 
		\]
		Then, for any $z =C_{p,\lambda,t} x + D_{p,\lambda,t}y$,  with $x^{1/\gamma} \in \text{supp}(f)$ and $y^{1/\gamma}  \in \text{supp}(g)$, the generalized generalized H\"older inequality (\ref{holder}) and  inequality \eqref{e:one} yield that 
		\begin{equation*}\label{e:two}
		\begin{split}
		w(z)z^{\frac{1}{\delta}} &\geq \left[C_{p,\lambda,t}u(x)^{\al} + D_{p,\lambda,t}v(y)^{\al}  \right]^{\frac{1}{\al}}\\
		&\times \left[C_{p,\lambda,t}(x^{1/\delta})^{\delta} + D_{p,\lambda,t}(y^{1/\delta})^{\delta} \right]^{\frac{1}{\delta}}\\
		&= AB[(1-\theta)u(x)^{\al} + \theta v(y)^{\alpha}]^{\frac{1}{\al}} \left[(1-\theta)(x^{1/\delta})^{\delta} + \theta (y^{1/\delta})^{\delta} \right]^{\frac{1}{\delta}}\\
		&\geq AB \left[(1-\theta)(u(x)x^{1/\delta})^{\alpha_0} + \theta (v(y)y^{1/\delta})^{\alpha_0} \right]^{\frac{1}{\alpha_0}}\\
		&= \left[C_{p,\lambda,t}(u(x)x^{1/\delta})^{\alpha_0} + D_{p,\lambda,t}(v(y)y^{1/\delta})^{\alpha_0} \right]^{\frac{1}{\alpha_0}}
		\end{split}
		\end{equation*}
		where $\alpha_0$ is defined by 
		\[
		\frac{1}{\alpha_0} = \frac{1}{\alpha} + \frac{1}{\delta} = \frac{1}{\alpha} + \frac{1}{\gamma} -1.
		\]
		
		Therefore, the triple 
		\begin{equation}\label{formula48}
		(w(z)z^{1/\delta}, u(x)x^{1/\delta}, v(y)y^{1/\delta})
		\end{equation}
		satisfy the conditions of the $L_p$-Borell-Brascamp-Lieb inequality (\ref{e:LpBBLassumption}), provided $\alpha_0 \geq -1$; in which case, we would have 
		\[
		\int_0^{\infty}w(z)z^{1/\delta} dz \geq \left((1-t)\left(\int_0^{\infty}u(x)x^{1/\delta} dx \right)^{\beta} + t \left(\int_0^{\infty}v(y)y^{1/\delta}dy \right)^{\beta} \right]^{\frac{1}{\beta}},
		\]
		where $\beta = \frac{p\alpha\gamma}{\alpha+\gamma}$. Finally, using the fact that \[
		\int_0^{\infty}u(x)x^{1/\delta} dx = \int_0^{\infty}f\left(x^{1/\gamma}\right)x^{1/\gamma-1} dx = |\gamma|\int_0^{\infty} f(x) dx,
		\]
		and the same with the pairs $(u,h)$, and $(v,g)$, we would have inequality \eqref{e:conclusion}, as desired.  Therefore (\ref{formula48}) concludes the inequality (\ref{e:conclusion}) of Theorem \ref{t:1dgeneralBorell-Brascamp-Lieb }, provided that
		\begin{enumerate}
			\item[(a)] $\alpha + \delta > 0$;
			\item [(b)] $\alpha_0 \geq -1$. 
		\end{enumerate}
		
		For the remain cases to $\gamma$, they have similar proofs for
		\textbf{\circled{3} The case $0<\gamma<1$,}		
		\textbf{\circled{4} The case $-\infty<\gamma<0$,}  and
		\textbf{\circled{4} The case $\gamma=-\infty$} in \cite[Page~19]{BCF} by using $L_p$ coefficients.
		
	\end{proof}
	
	\begin{remark}
		If $p=1$, it recovers the result of Theorem 4.1 in \cite{BCF}.
	\end{remark}
		
	In the following, we consider several consequences of Theorem~\ref{t:LpBorell-Brascamp-Lieb functionals} for certain choices of the functional $\widetilde{\Omega}$.
	The first consequence comes by choosing $\widetilde{\Omega}(\cdot) = \mu(\cdot)$ a $\alpha$-concave measure on $\R^n$ with $\alpha \geq -1$. We obtain a $L_p$-Borell-Brascamp-Lieb  type inequality for integrals of functions when integrated with respect to $\mu$. 
	
	\begin{corollary} Let $p,q \in[1,\infty]$ be such that  $1/p+1/q=1$. Suppose that $\alpha\geq -1$ and suppose that $\mu$ is an $\alpha$-concave measure on the class of Borel measurable subsets of $\R^n$ \emph{(respectively, $\mathcal{K}^n_{(o)}$)}. Let $\gamma \geq-\alpha$. Suppose that $h,f,g \colon \R^n \to \R_+$ are a triple of integrable Borel measurable (respectively, quasi-concave functions) that satisfy the condition 
		\begin{equation*}
		\begin{split}
		h\left(C_{p,\lambda,t}x + D_{p,\lambda,t}y  \right) \geq\left[C_{p,\lambda,t}f(x)^{\alpha} + D_{p,\lambda,t} g(y)^{\alpha} \right]^{\frac{1}{\alpha}}
		\end{split}
		\end{equation*}
		for every $x \in \text{supp}(f)$, $y \in \text{supp}(g)$, and $\lambda \in (0,1)$ whenever $f(x)g(y) >0$.  Then the following inequality holds:
		\begin{equation*}
		\begin{split}
		\int_{\R^n} h(x) d\mu(x) \geq \left[(1-t) \left(\int_{\R^n}f(x)d\mu(x) \right)^{\be} + t \left(\int_{\R^n}g(x)d\mu(x) \right)^{\be}\right]^{\frac{1}{\beta}}, \quad \be = \frac{p\alpha\gamma}{\alpha+\gamma}.
		\end{split}
		\end{equation*}
		
	\end{corollary}

	\subsection{New proofs of $L_p$-Borell-Brascamp-Lieb type inequality} \label{subsection32}
	The main goal of Subsections \ref{subsection32}  is to extend the $L_p$-Borell-Brascamp-Lieb inequality appearing in \cite{RX} for the range $s \geq 0$, to the range $[-\infty,\infty]$ using different methods of proof. Particularly, these proof process are more concise than our previous works in \cite{RX} using the level sets and $L_p$ Brunn-Minkowski inequality in geometric setting for $s\geq0$. Here we also include the case for $s<0$ to complement the $L_p$-Borell-Brascamp-Lieb inequality for $s$. The result reads as follows.
	\begin{theorem}\label{t:Lpbbl}	
		Let $p\geq1$, $-\infty<s < \infty$, and $t \in (0,1)$. Let $f,g,h \colon \R^n \to \R_+$ be a triple of bounded integrable functions. Suppose, in addition, that this triple satisfies the condition
		\begin{equation}\label{e:LpBBLassumption}
		h\left(C_{p,\lambda,t}x + D_{p,\lambda,t}y\right) \geq \left[C_{p,\lambda,t} f(x)^{s} + D_{p,\lambda,t}g(y)^{s}\right]^{\frac{1}{s}}
		\end{equation}
		for every $x \in \text{supp}(f)$, $y\in \text{supp}(g)$ and every $\lambda \in [0,1]$. Then the following integral inequality holds:
		\begin{equation}\label{e:Borell-Brascamp-Lieb 1con}
		I(h) \geq
		\begin{cases} 
		M_{\gamma_1}^{((1-t),t)}\left(I(f), I(g)\right), &\text{if } s \geq -\frac{1}{n},\\
		\min \left\{\left[C_{p,\lambda,t}\right]^{\frac{1}{\gamma}}I(f),\left[D_{p,\lambda,t}\right]^{\frac{1}{\gamma}}I(g) \right), &\text{if } s<- \frac{1}{n},
		\end{cases}
		\end{equation}
		for $0\leq \lambda\leq1$,	where $\gamma_1= p\gamma$ and $\gamma= \frac{s}{1+ns}$.
	\end{theorem}
	
	By the defitions of $L_{p,s}$ supremal-convolution, we conclude that
	\begin{equation}\label{e:Borell-Brascamp-Lieb 1con}
	I((1-t)\times_{p,s}f\oplus_{p,s}t\times_{p,s}g) \geq
	\begin{cases} 
	M_{\gamma_1}^{((1-t),t)}\left(I(f), I(g)\right), &\text{if } s \geq -\frac{1}{n},\\
	\min \left\{\left[C_{p,\lambda,t}\right]^{\frac{1}{\gamma}}I(f),\left[D_{p,\lambda,t}\right]^{\frac{1}{\gamma}}I(g) \right), &\text{if } s < - \frac{1}{n}.
	\end{cases}
	\end{equation}

	\textbf{ (i) Proof of $L_p$-Borell-Brascamp-Lieb type inequality for $s\in[-1/n, \infty)$ using mass transportation.} 
	As is known that the method of mass transportation is widely used in proving functional inequalities, such as the Pr\'ekopa-Leindler inequality and Borell-Brascamp-Lieb inequality in \cite{barthe1, barthe2, Gar1,Trans}, etc. Since the $L_p$-Borell-Brascamp-Lieb inequality includes the typical case for $s=0$---the Pr\'ekopa-Leindler inequality, and $p=1$---the Borell-Brascamp-Lieb inequality,  we attempt to using the mass transportation method to solve Theorem \ref{t:Lpbbl} and show that the case for $s\geq -1/n$ works in an analogous approach accordingly.	
	
	Before proving the theorem, we require the so-called Minkowski determinant inequality (see \cite{AGM}) for matrices. 
	
	\begin{lemma}\label{t:detinequality} Let $A,B$ be $n \times n$ positive symmetric semi-definite matrices, and $a,b \geq 0$. Then one has that 
		\[
		\det(aA + bB)^{\frac{1}{n}} \geq a \det(A)^{\frac{1}{n}} + b \det(B)^{\frac{1}{n}}.
		\]
		
	\end{lemma}
	
	\begin{proof}[\textbf{Proof of Theorem \ref{t:Lpbbl} for $s\geq-1/n$}] Without loss of generality, we may assume that $I(f), I(g) = 1$, and denote probability measures $\mu$ and $\nu$ defined on $\R^n$ satisfying $d\mu(y) = f(y) dy$ and $d\nu(y) = g(y)dy$. Suppose that $\rho$ is the uniform measure on $[0,1]^n$. Recall the proof  due to F. Barthe in \cite[Page~188-189]{Trans} relies on the concept of mass transportation.
		Since  $\mu, \nu$ are probability measure on $\R^n$ which are absolutely continuous with respect to the Lebesgue measure on $\R^n$ 
		there exist two convex functions $\phi_1, \phi_2 \colon \R^n\rightarrow \R$, whose gradient maps $\nabla \phi_1$  and $\nabla \phi_2$, respectively transport $\rho$ to $\mu$ and $\rho$ to $\nu$, i.e., $(\nabla \phi_1) \rho = \mu$ and $(\nabla \phi_2) \rho = \nu$. The change of variable formulas lead to the following results a.e. on $[0,1]^n$:
		\[
		f(\nabla \phi_1(x)) \det(Hess \ \phi_1(x)) = 1, \quad g(\nabla \phi_2(x)) \det(Hess \ \phi_2(x)) = 1,
		\]
		where $Hess \ \varphi_i,\  i\in\{1,2\}$ are the Aleksandrov Hessians defined a.e. and are symmetric non-negative semi-definite. 
		
		Fix any $\lambda\in[0,1]$, and set $\varphi_{\lambda}= C_{p,\lambda,t}\phi_1 + D_{p,\lambda,t} \phi_2.$ 
		By the change of variable, together with (\ref{holder}), (\ref{e:LpBBLassumption}) and Lemma \ref{t:detinequality},
		we see that 
		\begin{align*}
		&\int_{\R^n} h(y) dy \\&\geq \int_{[0,1]^n} h(y)dy\\
		&=\int_{[0,1]^n}h(\nabla \varphi_{\lambda}(x)) \det(Hess \ \varphi_{\lambda}(x)) dx\\
		&\geq \int_{[0,1]^n}h(C_{p,\lambda,t}\nabla \phi_1(x) + D_{p,\lambda,t} \nabla \phi_2(x))M_{\frac{1}{n}}^{(C_{p,\lambda,t},D_{p,\lambda,t})}\left(\det(Hess \ \phi_1(x)),\det(Hess \ \phi_2(x)) \right)dx\\
		&\geq \int_{[0,1]^n}M_s^{(C_{p,\lambda,t},D_{p,\lambda,t})}(f((\nabla \phi_1)(x)),g((\nabla \phi_2)(x)))\\
		&\quad\times M_{\frac{1}{n}}^{(C_{p,\lambda,t},D_{p,\lambda,t})}\left(\det(Hess \ \phi_1(x)),\det(Hess \ \phi_2(x))\right) dx\\
		&\geq \int_{[0,1]^n}M_{\frac{s}{1+ns}}^{(C_{p,\lambda,t},D_{p,\lambda,t})}\left(f((\nabla \phi_1)(x))\det(Hess \ \phi_1(x)), g((\nabla \phi_2)(x))\det(Hess \ \phi_2(x))\right) dx\\
		&= \int_{[0,1]^n}[C_{p,\lambda,t}  + D_{p,\lambda,t}]^{\frac{1+ns}{s}}dx.
		\end{align*}
		Therefore, as $\lambda$ is arbitrary in $[0,1]$, we conclude that
		\[
		\int_{\R^n}h(y)dy\geq\sup_{0\leq\lambda\leq1}[C_{p,\lambda,t}  + D_{p,\lambda,t}]^{\frac{1+ns}{s}}\geq1,
		\]
		where if $s\geq0$, we choose $\lambda=t$, and if $-1/n\leq s<0$, we apply the H\"older inequality $C_{p,\lambda,t}+D_{p,\lambda,t}\leq1$ for $p\geq1$, completing the proof.

	\end{proof}
	
	\textbf{ (ii) Proof of $L_p$-Borell-Brascamp-Lieb inequality using classical Borell-Brascamp-Lieb inequality.}
	In the following, we will give another proof of $L_p$-Borell-Brascamp-Lieb inequality in Theorem \ref{t:Lpbbl} for $s\in(-\infty,\infty)$ by applying classic Borell-Brascamp-Lieb (BBL) inequality, which is different from but a more concise proof than \cite{RX} for $s\geq0$. Firstly, for $s\leq -1/n$, we require the following lemma of the Borell-Brascamp-Lieb inequality in \cite[Lemma~3.3]{DancsUhrin}.
	\begin{lemma}
		\label{slessthan}
		Let $f,g:\R^n\rightarrow\R_+$ be integrable functions, $-\infty<s<-1/n$, and $0\leq t\leq1.$ Then
		\begin{equation}\label{slessn}
		\int_{\R^n}\sup_{z=(1-t)x+ty}\big[(1-t)f(x)^s+t g(y)^s\big]^{1/s}dz\geq\min\Big\{(1-t)^{n+1/s}I(f),t^{n+1/s}I(g)\Big\}.
		\end{equation}
	\end{lemma}
	Furthermore, we conclude from this lemma by the definition of supremal-convolution as
	\[
	I((1-t)\times_sf\oplus_s t\times_s g)\geq \min\Big\{(1-t)^{n+1/s}I(f),t^{n+1/s}I(g)\Big\},
	\]
	which complement the result for $s$ in (\ref{bblclassic}).
	
	\begin{proof}[\textbf{Proof of Theorem \ref{t:Lpbbl}}] 
		First we provide the proof for $s\geq -1/n$ using the classical Borell-Brascamp-Lieb inequality.
		Fix $t\in(0,1)$. For $\lambda\in[0,1]$, let   $$\bar{x}:=(\frac{1-\lambda}{1-t})^{\frac{1}{q}}x, \quad \bar{y}:=(\frac{\lambda}{t})^{\frac{1}{q}}y$$
		and
		\[
		\tilde{f}(\bar{x}):=(\frac{1-\lambda}{1-t})^{\frac{1}{qs}}f(x), \quad 	\tilde{g}(\bar{y}):=(\frac{\lambda}{t})^{\frac{1}{qs}}g(y).
		\]
		Then we have
		\begin{eqnarray}
		&&\!\!\!\!\!\!\!\!\!\!\!\!\!\!\!\!\!\!\!\!\!\!\!\!\int_{\R^n}\sup_{0\leq\lambda\leq1}	\sup_{z=(1-t)^{1/p}(1-\lambda)^{1/q}x+t^{1/p}\lambda^{1/q}y}[(1-t)^{1/p}(1-\lambda)^{1/q}f(x)^s+t^{1/p}\lambda^{1/q}g(y)^s]^{1/s}dz\nonumber\\
		\quad \ &\geq&\sup_{0\leq\lambda\leq1}\int_{\R^n}	\sup_{z=(1-t)^{1/p}(1-\lambda)^{1/q}x+t^{1/p}\lambda^{1/q}y}[(1-t)^{1/p}(1-\lambda)^{1/q}f(x)^s+t^{1/p}\lambda^{1/q}g(y)^s]^{1/s}dz\nonumber\\
		&=&\sup_{0\leq\lambda\leq1}\int_{\R^n}	\sup_{z=(1-t)[(\frac{1-\lambda}{1-t})^{\frac{1}{q}}x]+t[(\frac{\lambda}{t})^{\frac{1}{q}}y]}\big\{(1-t)[(\frac{1-\lambda}{1-t})^{\frac{1}{qs}}f(x)]^s+t[(\frac{\lambda}{t})^{\frac{1}{qs}}g(y)]^s\big\}^{1/s}dz\nonumber\\	
		&=&\sup_{0\leq\lambda\leq1}\int_{\R^n}	\sup_{z=(1-t)\bar{x}+t\bar{y}}(1-t)	[\tilde{f}(\bar{x})^s+t\tilde{g}(\bar{y})^s]^{1/s}dz\nonumber\\	
		&\geq& \sup_{0\leq\lambda\leq1}	\big((1-t)\{\int_{\R^n}	\tilde{f}(\bar{x})d\bar{x}\}^{\frac{s}{1+ns}}+	t\{\int_{\R^n}	\tilde{g}(\bar{y})d\bar{y}\}^{\frac{s}{1+ns}}\big)^{\frac{1+ns}{s}}\quad \quad \text{(by (\ref{bblclassic}), BBL inequality)}\nonumber\\	
		&=& \sup_{0\leq\lambda\leq1}	\big((1-t)\{\int_{\R^n}(\frac{1-\lambda}{1-t})^{\frac{1+ns}{qs}}f(\omega)d\omega\}^{\frac{s}{1+ns}}+	t\{\int_{\R^n}	(\frac{\lambda}{t})^{\frac{1+ns}{qs}}g(\gamma)d\gamma\}^{\frac{s}{1+ns}}\big)^{\frac{1+ns}{s}}\nonumber\\	
		&\geq&\sup_{0\leq\lambda\leq1}\big[(1-t)^{1/p}(1-\lambda)^{1/q}\big(I(f)\big)^{\frac{s}{1+ns}}+t^{1/p}\lambda^{1/q}\big(I(g)\big)^{\frac{s}{1+ns}}\big]^{\frac{1+ns}{s}}\label{lpbblformula0}\\
		&=&\big[(1-t)\big(I(f)\big)^{\frac{ps}{1+ns}}+t\big(I(g)\big)^{\frac{ps}{1+ns}}\big]^{\frac{1+ns}{ps}}\nonumber
		\end{eqnarray}
		where the last equality comes from Lemma \ref{supremallp} (1) for $s\geq0$,  and if $-1/n\leq s<0$, we use the fact that \begin{eqnarray*}&&\!\!\!\!\!\!\!\!\!\!\!\!\!\!\!\!\!\!\sup_{0\leq\lambda\leq1}\big[(1-t)^{1/p}(1-\lambda)^{1/q}\big(I(f)\big)^{\frac{s}{1+ns}}+t^{1/p}\lambda^{1/q}\big(I(g)\big)^{\frac{s}{1+ns}}\big]^{\frac{1+ns}{s}}\\
			&\geq& \inf_{0\leq\lambda\leq1}\big[(1-t)^{1/p}(1-\lambda)^{1/q}\big(I(f)\big)^{\frac{s}{1+ns}}+t^{1/p}\lambda^{1/q}\big(I(g)\big)^{\frac{s}{1+ns}}\big]^{\frac{1+ns}{s}}\end{eqnarray*} first together with Lemma \ref{supremallp} (1)  afterwards, as desired.

		For $s<-1/n$, by Lemma \ref{slessthan}, we  have
		\begin{eqnarray*}
			&&\int_{\R^n}\sup_{0\leq\lambda\leq1}	\sup_{z=(1-t)^{1/p}(1-\lambda)^{1/q}x+t^{1/p}\lambda^{1/q}y}[(1-t)^{1/p}(1-\lambda)^{1/q}f(x)^s+t^{1/p}\lambda^{1/q}g(y)^s]^{1/s}dz\\
			&\geq&\sup_{0\leq\lambda\leq1}\int_{\R^n}	\sup_{z=(1-t)\bar{x}+t\bar{y}}(1-t)	[\tilde{f}(\bar{x})^s+t\tilde{g}(\bar{y})^s]^{1/s}dz\\	
			&\geq&\sup_{0\leq\lambda\leq1} \min\Big\{(1-t)^{n+1/s}\int_{\R^n}\tilde{f}(\bar{x})d\bar{x},t^{n+1/s}\int_{\R^n}\tilde{g}(\bar{y})d\bar{y}\Big\}\\
			&=&\sup_{0\leq\lambda\leq1} \min\Big\{(1-t)^{n+1/s}\int_{\R^n}(\frac{1-\lambda}{1-t})^{\frac{1+ns}{qs}}f(\omega)d\omega,t^{n+1/s}\int_{\R^n}(\frac{\lambda}{t})^{\frac{1+ns}{q
					s}}g(\gamma)d\gamma\Big\}\\
			&=&\sup_{0\leq\lambda\leq1} \min\Big\{(1-\lambda)^{\frac{1+ns}{ps}}(1-t)^{\frac{1+ns}{qs}}\int_{\R^n}f(\omega)d\omega,\lambda^{\frac{1+ns}{ps}}t^{\frac{1+ns}{qs}}\int_{\R^n}g(\gamma)d\gamma\Big\}\\
			&\geq& \min\big\{C_{p,\lambda,t}^{\frac{1+ns}{s}}I(f), D_{p,\lambda,t}^{\frac{1+ns}{s}}I(g)\big\}
		\end{eqnarray*}
		for $0\leq \lambda\leq1$.
	\end{proof}
	\begin{remark}It can be checked easily that if $p=1$, it recovers the result of Lemma \ref{slessthan} and the classic Borell-Brascamp-Lieb inequality. Moreover, this method of proof to introduce $\tilde{f}(\bar{x})$ and $\tilde{g}(\bar{x})$ also works  in Theorem \ref{t:1dgeneralBorell-Brascamp-Lieb } but  only  for $n=1$ and $\gamma=1$. 
	\end{remark}

	\section{Applications of $L_{p}$-Borell-Brascamp-Lieb inequality}\label{section4}
	
	The goal of this section is to provide several functional analytic and measure theoretic consequences of the topics discussed in Section~\ref{section3}.  Based on the restriction conditions on  $L_p$-Borell-Brascamp-Lieb  type inequalities, we define the following concavity definitions in $L_p$ case for functions. It is inspired that if $h=f=g$ in the Borell-Brascamp-Lieb inequality condition, it recovers the $s$-concavity definition. Therefore, by letting $h=f=g$ in the $L_p$-Borell-Brascamp-Lieb inequality condition, we provide the $L_{p,s}$ concavity definitions.
	
	\begin{definition}\label{t:lpfunctions}
		Let $p \geq 1$, $1/p + 1/q = 1$, and $s \in [-\infty,+\infty]$. 
		
		\begin{enumerate}
				
			\item We say that a function $f\colon \R^n \to \R_+$ is $L_{p,s}$-concave if, for any pair $x,y \in \R^n$, one has 
			\begin{equation*}\label{e:lpfunctionalconditional}
			f\left(C_{p,\lambda,t}x +D_{p,\lambda,t}y \right) \geq M_{s}^{(C_{p,\lambda,t},D_{p,\lambda,t})}(f(x),f(y))
			\end{equation*}
			for every $\lambda \in [0,1]$ and $t\in[0,1]$. In this case,
			\[
			f(z) \geq \sup_{0\leq \lambda \leq1} \sup\left\{ M_{s}^{(C_{p,\lambda,t},D_{p,\lambda,t})}(f(x),f(y)) \colon z = C_{p,\lambda,t}x+D_{p,\lambda,t}y \right\}. 
			\]
			
			\item	Similarly, if $s=-\infty$, the function $f$ is said to be $L_{p}$-quasi-concave if, for any pair $x,y \in \R^n$, one has 
			\begin{equation*}\label{e:lpfunctionalquasi}
			f\left(C_{p,\lambda,t}x +D_{p,\lambda,t}y \right) \geq \min \left\{f(x),f(y) \right\}
			\end{equation*}
			for every $\lambda \in [0,1]$  and $t\in[0,1]$. 
			
			\item	If $s=0$, the function $f$ is said to be $L_{p}$-log-concave,  if for any pair $x,y \in \R^n$, one has 
			\begin{equation*}\label{e:lpfunctionalquasi}
			f\left(C_{p,\lambda,t}x +D_{p,\lambda,t}y \right) \geq 	f(x)^{C_{p,\lambda,t}}
			f(y)^{D_{p,\lambda,t}}	\end{equation*}
			for every $\lambda \in [0,1]$  and $t\in[0,1]$. 
			
			\item We call the function $f$ is said to be $L_{p,s}$-quasi-concave if, for any pair $x,y \in \R^n$, one has 
			\begin{equation*}\label{e:lpfunctionalquasi}
			f\left(C_{p,\lambda,t}x +D_{p,\lambda,t}y \right) \geq \min \left\{C_{p,\lambda,t}^sf(x),D_{p,\lambda,t}^sf(y) \right\}
			\end{equation*}
			for every $\lambda \in [0,1]$  and $t\in[0,1]$. 
			
		\end{enumerate}
	\end{definition}
	It is easy to see that  (4) recovers the definition of (2) if $s=0$, and it is inspired by the result of $L_p$ Borell-Brascamp-Lieb inequality for $s<-1/n$  in Theorem \ref{t:Lpbbl}.

	\begin{proposition}\label{t:concavitystrengthening}
		Let $p \geq 1$ and $s>0$.  If $f \colon \R^n \to \R_+$ is an $s$-concave function whose support contains the origin in its interior, then $f$ is also $L_{p,s}$-concave. 
	\end{proposition}
	
	\begin{proof}
		We only show the proof for $s \neq 0, \pm \infty$ as these cases are essentially identical.  Let $t,\lambda \in [0,1]$, $1/p+1/q=1$. Then, for any $x,y \in \R^n$ belonging to the support of $f$, we see that 
		\begin{align*}
		f\left(C_{p,\lambda,t}x + D_{p,\lambda,t}y \right) &= f\left(C_{p,\lambda,t} x +(1-C_{p,\lambda,t}) \frac{D_{p,\lambda,t}}{1-C_{p,\lambda,t}}y\right)\\
		&\geq \left[C_{p,\lambda,t}f(x)^{s} + (1-C_{p,\lambda,t}) f\left(\frac{D_{p,\lambda,t}}{1-C_{p,\lambda,t}}y\right)^{s} \right]^{\frac{1}{s}}\\
		&\geq M_{s}^{(C_{p,\lambda,t},D_{p,\lambda,t})}(f(x),f(y)),
		\end{align*}
		where in the last step we used the fact that the support of $f$ contains the origin in its interior together with H\"older's inequality, as required.
	\end{proof}

	We have similar definitions for measures with the $L_p$ coefficients. 
	
	\begin{definition}\label{t:lpmeasures}
		Let $p \geq 1$, $1/p + 1/q = 1$, and $s \in [-\infty,+\infty]$. We say that a non-negative measure $\mu$ on $\R^n$ is $L_{p,s}$-concave if, for any pair of Borel measurable sets $A,B \subset \R^n$, one has 
		\begin{equation*}\label{e:lpmeasureconditional}
		\mu\left(C_{p,\lambda,t}A+D_{p,\lambda,t}B\right) \geq M_s^{(C_{p,\lambda,t},D_{p,\lambda,t})}(\mu(A),\mu(B))
		\end{equation*}
		for every $\lambda \in [0,1]$ and $t\in[0,1]$.  Similarly, if $s =-\infty$, the measure $\mu$ is said to be $L_{p,s}$-quasi-concave if, for any pair of compact $A,B \subset \R^n$, one has 
		\begin{equation*}\label{e:lpmeasurequasi}
		\mu\left(C_{p,\lambda,t}A +D_{p,\lambda,t}B \right) \geq \min \left\{\mu(A),\mu(B) \right\}
		\end{equation*}
		for every $\lambda \in [0,1]$ and $t\in[0,1]$. Furthermore,  if $s =0$, the measure $\mu$ is said to be $L_{p,s}$-log-concave if, for any pair of compact $A,B \subset \R^n$, one has 
		\begin{equation*}\label{e:lpmeasurequasi}
		\mu\left(C_{p,\lambda,t}A +D_{p,\lambda,t}B \right) \geq \mu(A)^{C_{p,\lambda,t}}\mu(B)^{D_{p,\lambda,t}} 
		\end{equation*}
		for every $\lambda \in [0,1]$ and $t\in[0,1]$.
		Moreover, we call the measure $\mu$ is said to be $L_{p,s}$-quasi-concave if, for any pair of compact $A,B \subset \R^n$, one has 
		\begin{equation*}\label{e:lpfunctionalquasi}
		\mu\left(C_{p,\lambda,t}A +D_{p,\lambda,t}B \right) \geq \min \left\{C_{p,\lambda,t}^s\mu(A),D_{p,\lambda,t}^s\mu(B) \right\}
		\end{equation*}
		for every $\lambda \in [0,1]$  and $t\in[0,1]$.

	\end{definition}
	
	The next result concerns convolutions concavities related to the $L_{p,s}$-concave functions (see also \cite[Pages 643-644]{Uhrin2} for the case $p=1$). 
	
	\begin{theorem}\label{t:lpconcavemeasures} Let $p \geq 1$, $1/p + 1/q = 1$, $t \in [0,1]$, and $s, \beta \in [-\infty,+\infty]$ be such that $s + \beta \geq 0$. Let $f,g \colon \R^n \to \R_+$ be $L_{p,s}$-concave and $L_{p, \beta}$-concave, respectively. Then the convolution of $f$ and $g$,
		\[
		(f*g)(z) = \int_{\R^n}f(x)g(z-x)dx,
		\]
		satisfies one of the following: 
		\begin{enumerate}
			\item is $L_{p,(s^{-1} + \beta^{-1} + n)^{-1}}$-concave whenever $\frac{s \beta}{s + \beta} \in \left[-\frac{1}{n}, +\infty \right)$;
			
			\item is $L_{p,(s^{-1} + \beta^{-1} + n)}$-quasi-concave whenever $\frac{s \beta}{s + \beta} \in \left(-\infty, -\frac{1}{n} \right)$.
		\end{enumerate}
		
	\end{theorem}
	
	\begin{proof}
		Let $t\in [0,1]$. Since $f,g$ are $L_{p,s}$-concave and $L_{p, \beta}$-concave, respectively, the condition indicate that for fixed $v,w\in\R^n$,
		\begin{align*}
		&f\left(z\right) \geq \sup_{0 \leq \lambda \leq 1} \left[\sup_{z=C_{p,\lambda,t}x +D_{p,\lambda,t}y }M_{s}^{(C_{p,\lambda,t},D_{p,\lambda,t})}(f(x),f(y))\right],\\
		&g\left(C_{p,\lambda,t}v +D_{p,\lambda,t}w - z\right) \geq \sup_{0 \leq \lambda \leq 1} \left[\sup_{z=C_{p,\lambda,t}x +D_{p,\lambda,t}y }M_{\beta}^{(C_{p,\lambda,t},D_{p,\lambda,t})}(g(v-x),g(w-y))\right]. 
		\end{align*}
		Therefore, by  applying the generalized H\"older inequality, the formula (\ref{lpbblformula0}) for $\gamma=\frac{s\beta}{s+\beta}\geq -1/n$, and (\ref{slessn}) for $\gamma< -1/n$, we obtain 
		\begin{align*}
		&(f *g)\left(C_{p,\lambda,t}v +D_{p,\lambda,t}w\right)\\
		&= \int_{\R^n}f(z) g\left(C_{p,\lambda,t}v +D_{p,\lambda,t}w-z\right)dz\\
		&\geq \int_{\R^n}\sup_{0 \leq \lambda \leq 1} \left[\sup_{z=C_{p,\lambda,t}x +D_{p,\lambda,t}y }M_{s}^{(C_{p,\lambda,t},D_{p,\lambda,t})}(f(x),f(y))M_{\beta}^{(C_{p,\lambda,t},D_{p,\lambda,t})}(g(v-x),g(w-y))\right] dz\\
		&\geq \int_{\R^n}\sup_{0 \leq \lambda \leq 1} \left[\sup_{z=C_{p,\lambda,t}x +D_{p,\lambda,t}y }M_{\gamma}^{(C_{p,\lambda,t},D_{p,\lambda,t})}(f(x)g(v-x),f(y)g(w-y)) \right]dz\\
		&=\begin{cases}
		\sup_{0\leq\lambda\leq1}\big[C_{p,\lambda,t}\big((f*g)(v)\big)^{\gamma_0}+D_{p,\lambda,t}\big((f*g)(w)\big)^{\gamma_0}\big]^{\frac{1}{\gamma_0}}\label{lpbblformula}, &\text{if } \gamma \geq -\frac{1}{n},\\
		\min\left\{\left[C_{p,\lambda,t}\right]^{\frac{1}{\gamma_0}}(f*g)(v),\left[D_{p,\lambda,t}\right]^{\frac{1}{\gamma_0}} (f*g)(w)\right\}, &\text{if } \gamma < - \frac{1}{n},
		\end{cases}
		\end{align*}
		for all $0\leq\lambda\leq1$,
		where $\gamma_0= \frac{\gamma}{1+n\gamma}=(s^{-1} + \beta^{-1}+n)^{-1}$. Therefore,
		\begin{eqnarray}
		&&\!\!\!\!\!\!\!\!\!\!\!\!\!\!\!\!\!\!(f *g)\left(C_{p,\lambda,t}v +D_{p,\lambda,t}w\right)\nonumber\\
		&\geq&
		\begin{cases}
		\big[C_{p,\lambda,t}\big((f*g)(v)\big)^{\gamma_0}+D_{p,\lambda,t}\big((f*g)(w)\big)^{\gamma_0}\big]^{\frac{1}{\gamma_0}}\label{lpbblformula}, &\text{if } \gamma \geq -\frac{1}{n},\\
		\min\left\{\left[C_{p,\lambda,t}\right]^{\frac{1}{\gamma_0}}(f*g)(v),\left[D_{p,\lambda,t}\right]^{\frac{1}{\gamma_0}} (f*g)(w)\right\}, &\text{if } \gamma < - \frac{1}{n},
		\end{cases}
		\end{eqnarray}	
		for all $0\leq\lambda\leq1$. Therefore, $f*g$  is $L_{p,(s^{-1} + \beta^{-1} + n)^{-1}}$-concave whenever $\frac{s \beta}{s + \beta} \in \left[-\frac{1}{n}, +\infty \right)$, and is $L_{p,(s^{-1} + \beta^{-1} + n)}$-quasi-concave whenever $\frac{s \beta}{s + \beta} \in \left(-\infty, -\frac{1}{n} \right)$.
	\end{proof}
	
	By the series of $L_{p,s}$ concavity definitions, we deduce from Theorem \ref{t:Lpbbl} and formula (\ref{lpbblformula0}) that, if a measure has a density that is $L_{p,s}$-concave for $s\geq-1/n$, then the measure itself is $L_{p,\frac{s}{1+ns}}$-concave, and $L_{p,\frac{1+ns}{s}}$-quasi-concave for $s<-1/n$. Therefore, we have the following extension of the $L_p$ version of Brunn's concavity principle (see \cite{AGM} and \cite{Pivov} for $p=1$). 
	
	\begin{corollary}\label{t;convolution}
		Let $K \subset \R^n$ be a convex body containing the origin in its interior, $H$ is a $(n-j)$-dimensional subspace of $\R^n$, and $j\in\{0,\cdots,n-1\}$. Let $\mu$ be a measure on $\R^n$ whose density is $L_{p,s}$-concave for some $s \in [-\infty,+\infty]$; i.e., $d\mu(x)/dx=f(x)$ and $f(x)$ is $L_{p,s}$-concave. The function $\Omega \colon H \to \R_+$ given by 
		\[
		\Omega(x)= \mu(K \cap (x+H)), \quad x \in H
		\]
		satisfies
		\begin{enumerate}
			\item $\Omega$ is a $L_{p,\gamma}$-concave function on its support for $s\geq -\frac{1}{n-j}$;
			\item  $\Omega$ is a $L_{p,\frac{1}{\gamma}}$-quasi-concave function on  its support for $s<-\frac{1}{n-j}$
		\end{enumerate}
		where $\gamma= \frac{s}{1+(n-j)s}$.
	\end{corollary}
	
	Another $L_{p,s}^{\gamma}$ concavity definition only works in 1-dimension space $\R$ by the restriction of parameter $\gamma$, which is not applicable for measures either. Recall the  condition in $L_{p,\gamma}$ Borell-Brascamp-Lieb inequality in $\R$	(\ref{e:assumption1}), that is,
	\begin{equation}
	\begin{split}
	h\left(\left(C_{p,\lambda,t}x^{\ga} + D_{p,\lambda,t}y^{\ga} \right)^{\frac{1}{\ga}} \right) \geq\left[C_{p,\lambda,t}f(x)^{\alpha} + D_{p,\lambda,t} g(y)^{\alpha} \right]^{\frac{1}{\alpha}},
	\end{split}
	\end{equation}
	we define the following concavity definitions.
	\begin{definition}\label{t:lpfunctions}
		Let $p \geq 1$, $1/p + 1/q = 1$, and $s \in [-\infty,+\infty]$. 
		
		\begin{enumerate}
			
			\item We say that a function $f\colon \R \to \R_+$ is $L_{p,s}^{\gamma}$-concave if, for any pair $x,y \in \R$, one has 
			\begin{equation*}\label{e:lpfunctionalconditional}
			f\left(\left(C_{p,\lambda,t}x^{\ga} + D_{p,\lambda,t}y^{\ga} \right)^{\frac{1}{\ga}}  \right) \geq M_{s}^{(C_{p,\lambda,t},D_{p,\lambda,t})}(f(x),f(y))
			\end{equation*}
			for every $\lambda \in [0,1]$ and $t\in[0,1]$. 
			
			\item	Similarly, if $s=-\infty$, the function $f$ is said to be $L_{p}^{\gamma}$-quasi-concave if, for any pair $x,y \in \R$, one has 
			\begin{equation*}\label{e:lpfunctionalconditional}
			f\left(\left(C_{p,\lambda,t}x^{\ga} + D_{p,\lambda,t}y^{\ga} \right)^{\frac{1}{\ga}}  \right) \geq \min(f(x),f(y))
			\end{equation*}
			for every $\lambda \in [0,1]$ and $t\in[0,1]$. 
			
			\item	If $s=0$, the function $f$ is said to be $L_{p}^\gamma$-log-concave if, for any pair $x,y \in \R$, one has 
			\begin{equation*}\label{e:lpfunctionalconditional}
			f\left(\left(C_{p,\lambda,t}x^{\ga} + D_{p,\lambda,t}y^{\ga} \right)^{\frac{1}{\ga}}  \right) \geq f(x)^{C_{p,\lambda,t}}f(y)^{D_{p,\lambda,t}}
			\end{equation*}
			for every $\lambda \in [0,1]$ and $t\in[0,1]$. 
			
			\item We call the function $f$ is said to be $L_{p,s}^\gamma$-quasi-concave if, for any pair $x,y \in \R$, one has 
			\begin{equation*}\label{e:lpfunctionalconditional}
			f\left(\left(C_{p,\lambda,t}x^{\ga} + D_{p,\lambda,t}y^{\ga} \right)^{\frac{1}{\ga}}  \right) \geq \min(C_{p,\lambda,t}^sf(x),D_{p,\lambda,t}^sf(y))
			\end{equation*}
			for every $\lambda \in [0,1]$ and $t\in[0,1]$. 	
		\end{enumerate}
	\end{definition}
	It is easy to see that $L_{p,s}^{\gamma}$ coincides with $L_{p,s}$ concavity when $\gamma=1$ and $n=1.$	
	
	\section{ Integral  representation of $L_{p,s}$ mixed quermassintegral  for functions}\label{section5}
	
	In this section, we mainly focus on the extension of $L_p$ Brunn-Minkowski theory including mixed $p$-quermassintegrals and their integral representation formulas for convex bodies in \cite{Lutwak2} to the space of $\mathcal{F}_s(\R^n)$ endowed with the $L_{p,s}$ summations introduced in Section \ref{section2}. Therefore, we analyze
	the properties of projection for functions and $L_{p,s}$ supremal-convolution in Subsection \ref{subsection51} and for $L_{p,s}$ Asplund summation  in Subsection \ref{subsection52}, respectively. In conclusion, we obtain the integral representation of $L_{p,s}$ mixed quermassintegral for functions via variation formula of $L_{p,s}$ Asplund summation. This works as  it is reasonable to take the first variation formula with the linear coefficients for $L_p$ mean of base functions and Legendre transformation similar to $L_p$ mean of support functions for convex bodies in (\ref{variationconvexbody}).

	To begin with, recall the following classes of functions:
	\[
	\mathcal{F}_s(\R^n) = \left\{f \colon \R^n \to \R_+, f \text{ is } s\text{-concave}, \text{u.s.c}, f \in L^1(\R^n), f(o) = \|f\|_{\infty}>0 \right\},
	\]
	\[
	C_s(\R^n) = \left\{u \colon \R^n \to \R_+\cup \{+\infty\}, u \text{ is convex, l.s.c}, u(o)=0, \lim_{x \to \infty} \frac{u(x)}{\|x\|} = +\infty\right\}.
	\]
	
	\subsection{Projection for functions and $L_{p,s}$ supremal-convolution}\label{subsection51}
	
	Using a geometry point of view---the epigraph and subgraph of a function $f: \, \R^{n} \to \R$, we can see the $L_{p,s}$ supremal-convolution satisfy elegant geometric properties for its related graphs.
	Consider two sets in $\R^{n+1}$
	$$
	\epi \, f = \{(x,t)  \in \R^n \times \R:f(x) \leq t\}, \qquad \sub \, f = \{(x,t) \in \R^n \times \R: f(x) \geq t\},
	$$
	we have the following property by using 
	$Epi f$ for convex function (open up) and $Sub f$ for concave function (open down) $f$ correspondingly.
	
	\begin{proposition}\label{propertyforlpsummation}
		For $f,g\in\mathcal{F}_s(\R^n)$ and $s\in[-\infty,\infty]$, we have
		\begin{enumerate}	
			\item	
			
			$
			\epi\, (f \oplus_s g)^s = \epi \, (f^s) +\epi \, (g^s),\quad s<0;
			$
			
			\!\!\!\!\!\!$
			\sub\, (f \oplus_s g)^s = \sub \, (f^s) +\sub \, (g^s),\quad s\geq0.
			$
			
			\item
			$
			\epi \, \Big((\alpha \times_s f)^s\Big) = \alpha \cdot \epi \, (f^s), \quad s<0;
			$
			
			\!\!\!\!\!\!$
			\sub \, \Big((\alpha \times_s f)^s\Big) = \alpha \cdot \sub \, (f^s), \quad s\geq0.
			$	
		\end{enumerate}	
	Here $``+"$ is the classic Minkowski sum for sets in $\R^{n+1}$.
	\end{proposition}
	
	\begin{proof}
		
		(1)
		Note that for $s\geq0$ and an $s$-concave function $f$,  $Sub \, f^s$ is a convex set in $\R^{n+1}$ and $\epi \, (-f^s) = A_{({n+1})\times{(n+1)}} \left(\sub \, f^s \right)$, where $A_{({n+1})\times{(n+1)}} $ is the reflection matrix satisfying $A_{({n+1})\times{(n+1)}} (x_1,x_2,\cdots,x_n,x_{n+1})=(x_1,x_2,\cdots,x_n,-x_{n+1})$ for any ($n+1$)-dimensional vector $(x_1,x_2,\cdots,x_n,x_{n+1})\in\R^{n+1}$. That is,
		\begin{equation*}
		A_{(n+1)\times(n+1)}= 
		\begin{pmatrix}
		1 & 0 & \cdots & 0&0 \\
		0 & 1 & \cdots& 0&0 \\
		\vdots  & \vdots  & & \vdots & \vdots  \\
		0 & 0 & \cdots &1 & 0 \\
		0 & 0 & \cdots &0 & -1 
		\end{pmatrix} \in O(n+1), \qquad A_{(n+1)\times(n+1)}^2=I_{(n+1)\times(n+1)},
		\end{equation*}
		where $I_{(n+1)\times(n+1)}$ is the identity matrix.
		For $s<0$,  we have by the definition of supremal-convolution and formula (\ref{epigraph}) that
		\begin{align*}
		\epi \, ((f \oplus_s g)^s) &= \big \{(x,t)\in \R^n \times \R: \big\{\sup \limits_{x = x_1 + x_2} [f^s(x_1) + g^s(x_2)]^{1/s}\big\}^s \leq t \big \}\\
		&= \big \{(x,t)\in \R^n \times \R: \, \inf \limits_{x = x_1 + x_2} \left(f^s(x_1)+g^s(x_2)\right) \leq t \big \} \\
		&= \big \{(x,t)\in \R^n \times \R:  \Big[(f^s) \square (g^s)\Big](x)\leq t \big \}\\
		&=\epi \, \Big( (f^s) \square (g^s)\Big) \\
		&= \epi \, (f^s) + \epi(g^s).
		\end{align*}
		
		Then, for $s$-concave functions $f, g \geq 0$ and $s\geq0$, one has
		\begin{align*}
		A_{(n+1)\times(n+1)} \left( \sub\, ((f \oplus_s g)^s) \right)
		&= \epi \, (-(f \oplus_s g)^s) \\
		&= \big \{(x,t)\in \R^n \times \R: -\sup \limits_{x = x_1 + x_2} \left(f^s(x_1) + g^s(x_2)\right) \leq t \big \}\\
		&= \big \{(x,t)\in \R^n \times \R: \, \inf \limits_{x = x_1 + x_2} \left(-f^s(x_1) - g^s(x_2)\right) \leq t \big \} \\
		&= \big \{(x,t)\in \R^n \times \R:  \Big[(-f^s) \square (-g^s)\Big](x)\leq t \big \}\\
		&=\epi \, \Big( (-f^s) \square (-g^s)\Big) \\
		&= \epi \, (-f^s) + \epi(-g^s)\\
		&= A_{(n+1)\times(n+1)} \Big(\sub \, f^s \Big) + A_{(n+1)\times(n+1)}\Big(\sub \, g^s \Big).
		\end{align*}
		Hence,
		$
		\sub\, (f \oplus_s g)^s = \sub \, (f^s) +\sub \, (g^s).
		$

		(2) The proofs for $s\geq 0$ and $s<0$ follow naturally from (1) in similar lines.
	\end{proof}
	
	Next, we consider the definition for the projection of $s$-concave functions \cite{Klartag,Rock1} $f\in\mathcal{F}_s(\R^n)$ onto the ($n-j$)-dimensional subspace $H\in G_{n,n-j}$  as
	$$
	f_H(z)=\big(P_Hf\big)(z):= \sup \limits_{y \in H^{\perp}} f(z+y),  \qquad f\in\mathcal{F}_s(\R^n),
	$$
	and  the projection of convex base function  \cite{Klartag} $f\in C_s(\R^n)$ onto the ($n-j$)-dimensional subspace $H$  as
	$$
	u_H(x)=\big(\tilde{P}_H u\big)(x) = \inf \limits_{y \in H^{\perp}} u(x+y), \qquad u \in C_s(\R^n).
	$$
	
	Here we list some elegant properties for the above definitions of projections for functions with the supremal-convolution. Recall that in \cite{AAM},  $\sub (P_H f)=(\sub f)|\bar{H}$ for $s\geq 0$ and $\epi (P_H f)=(\epi f)|\bar{H}$ for $s<0$. Here $\bar{H}=span\{H, e_{n+1}\}$, where $H\in G_{n,n-j}$ is the Grassmannian manifold on $\R^n$ with the orthonormal basis $\{e_1,\cdots, e_n\}$ and $e_{n+1}\perp \R^{n}$ is a unit vector.
	
	\begin{proposition}\label{remark1} For any functions $f,g \in \mathcal{F}_s(\R^n)$, $j\in\{0, \ldots,n-1\}$ and $H\in G_{n,n-j}$, we have the following identities.
		\begin{enumerate}
			\item  
			$$ \label{pr:power}
			P_H (f^s) = (P_Hf)^s,  \qquad s>0;
			$$
			$$ \label{pr:power}
			\tilde{P}_H (f^s) = (P_Hf)^s, \qquad s<0;
			$$		
			$$ \label{pr:power}
			P_H (\log f) = \log (P_Hf), \qquad s=0.
			$$	
			\item $	P_H (\alpha\times_sf) = \alpha\times_s(P_Hf), s\in[-\infty,\infty].$
			
			\item $
			P_H \Big(f \oplus_{p,s} g \Big) = P_Hf \oplus_{p,s}P_Hg,  \quad s\in[-\infty,\infty], \quad p\geq 1.$
					
		\end{enumerate}
	\end{proposition}
	\begin{proof}
		(1) It is easy to see that for $s>0$, we have
		\[
		P_H (f^s)(z)=\sup \limits_{y \in H^{\perp}} f^s(z+y)=[\sup \limits_{y \in H^{\perp}} f(z+y)]^s=[P_H (f)(z)]^s;
		\]
		for $s<0$,
		\[
		\tilde{P}_H (f^s)(z)=\inf \limits_{y \in H^{\perp}} f^s(z+y)=[\sup \limits_{y \in H^{\perp}} f(z+y)]^s=[P_H (f)(z)]^s;
		\]
		for $s=0$,
		\[
		P_H (\log f)(z)=\sup \limits_{y \in H^{\perp}} \log f(z+y)=\log[\sup \limits_{y \in H^{\perp}} f(z+y)]=\log [P_H (f)(z)].
		\]
		
		(2) By the definition of supremal-convolution, we have
		\begin{eqnarray*}
			P_H(\alpha\times_s f)=P_H(\alpha^sf(\frac{x}{\alpha}))=\sup_{z\in H^{\perp}}\alpha^s f(\frac{x}{\alpha}+z)=\alpha^sP_Hf(\frac{x}{\alpha})=\alpha\times_sP_Hf(x),
		\end{eqnarray*}	
		as desired.
		
		(3) For $p\geq1$, $j\in\{0,\cdots, n-1\},$ and a subspace $H \subset G_{n,n-j}$, we denote $\bar{H} = span\big\{H, e_{n+1}\big\}$, where $e_{n+1}\perp H$. Then, for $s>0$, we obtain	
		\begin{eqnarray*}
			\sub \, \Big( P_H f^s\Big)&=&
			\sub \, \Big( f^s\Big)|\bar{H}\\
			&=&A_{(n-j+1)\times(n-j+1)} \Big(\epi \, (-f^s) | \bar{H} \Big) \\
			&=& A_{(n-j+1)\times(n-j+1)} \Big(\epi \, (-f^s) \Big) | \bar{H}\\
			&=&\sub \, (f^s) | \bar{H},
		\end{eqnarray*}
		where
		\begin{equation*}
		A_{(n-j+1)\times(n-j+1)}= 
		\begin{pmatrix}
		1 & 0 & \cdots & 0&0 \\
		0 & 1 & \cdots& 0&0 \\
		\vdots  & \vdots  &  & \vdots &\vdots \\
		0 & 0 & \cdots &1 & 0 \\
		0 & 0 & \cdots &0 & -1 
		\end{pmatrix} \in O(n-j+1).
		\end{equation*}
		In particular, Proposition 
		\ref{propertyforlpsummation} (1) and Proposition \ref{remark1} (1) imply
		\begin{eqnarray*}
		\sub \, \Big(P_H(f \oplus_s g)^s \Big) &=&\sub \, \Big((f \oplus_s g)^s \Big)|\bar{H}\\
		 &=&\Big(\sub(f^s)+\sub(g^s)\Big)|\bar{H}\\	
		 &=&\sub(f^s)|\bar{H}+\sub(g^s)|\bar{H}\\		
		&=& \sub \, (P_H(f^s)) + \sub \, (P_H(g^s)) \\
			&=& \sub \, ((P_Hf)^s) + \sub \, ((P_Hg)^s) \\
		&=& \sub \, \Big((P_Hf \oplus_s P_Hg)^s\Big).
		\end{eqnarray*}
		Hence,
		\begin{equation}\label{sumprojection}
		P_H(f \oplus_s g) = P_Hf \oplus_s P_Hg, \quad s >0.
		\end{equation}
		
		Now, by Proposition \ref{propertyforlpsummation},  (\ref{sumprojection}) and definition of supremal-convolution, we have for $s>0$,
		\begin{align*}
		\sub \, \Big[\Big(P_Hf \oplus_{p,s}P_Hg\Big)^s\Big] 
		&= \sub \Big((\sup \limits_{0 \leq \lambda \leq 1} (1-\lambda)^{\frac{1}{q}} \times_s P_Hf \oplus_s \lambda^{\frac{1}{q}} \times_s P_Hg)^s\Big)\\
		&=\bigcup_{0\leq\lambda\leq1} \sub \Big(( ((1-\lambda)^{\frac{1}{q}} \times_s P_Hf) \oplus_s (\lambda^{\frac{1}{q}} \times_s P_Hg))^s\Big) \\
		&=\bigcup_{0\leq\lambda\leq1} \sub \Big(( (1-\lambda)^{\frac{1}{q}} \times_s P_Hf)^s + (\lambda^{\frac{1}{q}} \times_s P_Hg)^s\Big)\\
		&=\bigcup_{0\leq\lambda\leq1} \Big( (1-\lambda)^{\frac{1}{q}}  \cdot\sub\, (P_Hf)^s + \lambda^{\frac{1}{q}}  \cdot\sub\, (P_Hg)^s\Big)\\
		&=\bigcup_{0\leq\lambda\leq1} \Big( (1-\lambda)^{\frac{1}{q}} \cdot (\sub\, f^s | \bar{H}) + \lambda^{\frac{1}{q}}  \cdot(\sub\, g^s \big| \bar{H})\Big)\\
		&=\bigg(\bigcup_{0\leq\lambda\leq1} \Big( (1-\lambda)^{\frac{1}{q}} \cdot(\sub\, f^s) + \lambda^{\frac{1}{q}}  \cdot(\sub\, g^s)\Big)\bigg) \Big| \bar{H}\\
		&=\bigg(\bigcup_{0\leq\lambda\leq1} \Big( \sub\, \{ (1-\lambda)^{\frac{1}{q}} \times_s f\}^s +\sub\, \{ \lambda^{\frac{1}{q}} \times_s g\}^s \Big)\bigg) \Big| \bar{H}\\
		&=\bigg(\bigcup_{0\leq\lambda\leq1} \Big(\sub\, \big\{ (1-\lambda)^{\frac{1}{q}} \times_s f \oplus_s \lambda^{\frac{1}{q}} \times_s g\big\}^s \Big)\bigg) \Big| \bar{H}\\
		&=\bigg(\sub(\sup_{0 \leq \lambda \leq 1} \Big(\, \big\{ (1-\lambda)^{\frac{1}{q}} \times_s f \oplus_s \lambda^{\frac{1}{q}} \times_s g\big\}^s \Big)\bigg) \Big| \bar{H}\\
		&=\Big(\sub \, (f \oplus_{p,s} g)^s \Big)\Big|\bar{H} \\
		&= \sub \, \Big(P_H\big\{f \oplus_{p,s} g \big \}^s\Big),
		\end{align*}
		as projection is distributive over set union operation. 
		
		For $s<0$, we only need to replace ``$\sub$" by ``$\epi$", then the proof follows in similar lines by Proposition \ref{propertyforlpsummation} and Proposition \ref{remark1}. For $s=0$, change $f^s=\log f$, and the formulas holds in a similar method. Therefore, one has
		$
		\bigg[P_H \Big(f \oplus_{p,s} g \Big)\bigg]^s = P_H \big(f \oplus_{p,s} g\big)^s = (P_Hf \oplus_{p,s}P_Hg)^s;
		$
		i.e.,
		\begin{equation*} \label{pr:proj}
		P_H \Big(f \oplus_{p,s} g \Big) = P_Hf \oplus_{p,s}P_Hg,  \quad p\geq 1. 
		\end{equation*}
	\end{proof}
	Moreover, it is easy to check that for $u\in C_s(\R^n)$, one has $$
	P_H\left[\big(1-s u(x)\big)_+^{\frac{1}{s}}\right] = \big(1 - s \tilde{P}_H u(x)\big)_+^{\frac{1}{s}}, \quad s\in[-\infty,\infty].
	$$	
	\subsection{Projection for function and $L_{p,s}$ Asplund summation}\label{subsection52}
	
	In this part, we examine the properties of projection functions and $L_{p,s}$ Asplund summation.
	We begin with the following proposition which demonstrates that the $L_p$ addition of convex functions for $p \geq 1$ is stable under projections given by (\ref{e:lpconvexadditionp}). This paves the way to compute the variation formula for quermassintegral for functions, i.e., the integral representation of $L_{p,s}$ mixed quermassintegral shown in Subsection \ref{subsection53}. 
	
	\begin{proposition}\label{projectionforbase} Let $p \geq 1  $, $u,v \in C_s(\R^n) $, and $\alpha, \beta \geq0$.  Then, for any $H \in G_{n,n-j}$, $j\in\{0,1,\cdots,n-1\}$, one has
		\[
		[(\alpha \boxtimes_p u) \boxplus_p ( \beta \boxtimes_p v)]_H = [\alpha \boxtimes_p u_H] \boxplus_p [\beta \boxtimes_p v_H]. 
		\]
	\end{proposition}
	
	\begin{proof}
		To begin with, we consider the epigraphs of $u$ and $v$.   Let  $\{e_,\dots,e_n,e_{n+1}\}$ be the canonical basis  on $\R^{n+1}$ and set $\bar{H} = \text{span}(H,e_{n+1})$ a $(n-j+1)$-dimensional space for $H \in G_{n,n-j}$. Then by the fact that  in \cite{AAM},  $\sub (P_H f)=(\sub f)|\bar{H}$ for
		$s\geq 0$ and $\epi (P_H f)=(\epi f)|\bar{H}$ for $s<0$, we obtain by (\ref{epigraph}) that
		\begin{eqnarray}
		\epi([\alpha\times u \square \beta\times v]_H) 
		&=&\nonumber \epi(\alpha\times u \square \beta\times v)|\bar{H}\\
		&=&\nonumber[\alpha \epi(u) + \beta \epi(v)] | \bar{H}\\
		&=& \nonumber\alpha \epi(u)|\bar{H} + \beta \epi(v)|\bar{H}\\
		&=&\nonumber\alpha \epi(u_H) + \beta \epi(v_H)\\
		&=& \epi(\alpha\times u_H \square \beta\times v_H).\label{formula54}
		\end{eqnarray}
		Therefore, we have that 
		\[
		[\alpha\times u \square \beta\times v]_H=\alpha\times u_H \square \beta\times v_H.
		\]
		Finally, observe that by (\ref{formula54}) and Lemma \ref{supremallp} (1), one has
		\begin{align*}
		[(\alpha \boxtimes_p u) \boxplus_p ( \beta \boxtimes_p v)]_H(x) &= \inf_{y \in x + H^{\perp}}[(\alpha \boxtimes_p u) \boxplus_p ( \beta \boxtimes_p v)](y)\\
		&= \inf_{y \in x + H^{\perp}} \left[\left(\alpha (u^*(y))^p + \beta (v^*(y))^p\right)^{\frac{1}{p}}\right]^*\\
		&=\inf_{y \in x + H^{\perp}} \left[\sup_{0 \leq \lambda \leq 1}\left\{ \alpha^{\frac{1}{p}} (1-\lambda)^{\frac{1}{q}}u^*(y) + \beta^{\frac{1}{p}}\lambda^{\frac{1}{q}}v^*(y) \right\}\right]^*\\
		&=\inf_{y \in x + H^{\perp}} \inf_{0\leq\lambda\leq1} \left[ \alpha^{\frac{1}{p}} (1-\lambda)^{\frac{1}{q}}u^*(y) + \beta^{\frac{1}{p}}\lambda^{\frac{1}{q}}v^*(y)\right]^*\\
		&=\inf_{0 \leq \lambda \leq 1}\left[\alpha^{\frac{1}{p}} (1-\lambda)^{\frac{1}{q}}\times u \square \beta^{\frac{1}{p}}\lambda^{\frac{1}{q}}\times v\right]_H(x)\\
		&=\inf_{0 \leq \lambda \leq 1}\left[\alpha^{\frac{1}{p}} (1-\lambda)^{\frac{1}{q}}\times u_H \square \beta^{\frac{1}{p}}\lambda^{\frac{1}{q}}\times v_H\right](x)\\
			&=\inf_{0 \leq \lambda \leq 1}\big[\alpha^{\frac{1}{p}} (1-\lambda)^{\frac{1}{q}}u_H^*+ \beta^{\frac{1}{p}}\lambda^{\frac{1}{q}} v_H^*\big]^*(x)\\
				&=\big[\sup_{0 \leq \lambda \leq 1}\alpha^{\frac{1}{p}} (1-\lambda)^{\frac{1}{q}}u_H^*+ \beta^{\frac{1}{p}}\lambda^{\frac{1}{q}} v_H^*\big]^*(x)\\
		&=\left[\big(\alpha(u_H^*(x))^p + \beta(v_H^*(x))^p\big)^{\frac{1}{p}} \right]^*(x)\\
		&=:[(\alpha \boxtimes_p u_H) \boxplus_p (\beta \boxtimes_p v_H)](x),
		\end{align*}
		completing the proof. 
	\end{proof}	
	\subsection{Variation formula of general quermassintegral for functions  and $p\geq1$}\label{subsection53}
	Next we consider the ``\textit{$L_{p,s}$ mixed quermassintegral}" of two functions $f,g \in \mathcal{F}_s(\R^n)$. This is based on the $p$-mixed quermassintegral definition for convex bodies in Lutwak's work \cite{Lutwak1}. First, we give the definition of quermassintegeral for functions.
	\begin{definition}
		The $j$-th \textit{quermassintegral} of function $f=(1-su)_+\in\mathcal{F}_s(\R^n)$ and $u\in C_s(\R^n)$ for $j\in\{0,\cdots,n-1\}$, is defined as
		$$
		W_j(f):= c_{n,j} \int \limits_{G_{n,n-j}}\, \int \limits_H P_H f(x) dx\, d\nu_{n,n-j}(H)= c_{n,j} \int \limits_{G_{n,n-j}}  J_s(\tilde{P}_Hu) d \nu_{n,n-j}(H).
		$$
	\end{definition}
	For each function $f \in \mathcal{F}_s(\R^n)$, and any $j\in \{0,\dots, n-1\}$, an application of Fubini's theorem yields the following
	\begin{equation*}\label{e:equivalentformulation}
	\begin{split}
	W_{j}(f) &=c_{n,j}\int \limits_{G_{n,n-j}}\, \int \limits_H P_H f(x) dx\, d\nu_{n,n-j}(H)\\
	&= c_{n,j}\int \limits_{G_{n,n-j}}\, \int \limits_0^{\infty} \V_{n-j} (\{x \colon P_H f(x) \geq t\}) dt\, d\nu_{n,n-j}(H)\\
	&= c_{n,j}\int \limits_{G_{n,n-j}}\, \int \limits_0^{\infty} \V_{n-j} (\{x \colon f(x) \geq t\}|H) dt\, d\nu_{n,n-j}(H)\\
	&=   \int \limits_0^{\infty}c_{n,j} \int \limits_{G_{n,n-j}}\ \V_{n-j} (\{x \colon f(x) \geq t\}|H)d\nu_{n,n-j}(H)\ dt\\
	&=\int \limits_0^{\infty} W_j(\{f \geq t\}) dt.
	\end{split}
	\end{equation*}
	Therefore, the quantity $W_j(f)$ can be interpreted in terms of the usual quermassintegrals of its super-level sets, which was originally considered in \cite{BCF}. We remark that several works on quermassintegrals for functions have appeared in the literature, for example, see \cite{BCF,Chen1,Chen2,MilRot}.
	
	Next, we may choose $\Omega(K) = W_j(K)$ in Theorem~\ref{t:LpBorell-Brascamp-Lieb functionals}, for $K \in \mathcal{K}_{(o)}^n$ and $j \in \{0,1,\dots,n-1\}$. The Brunn-Minkowski inequality for $W_j(\cdot)$, together with H\"older's inequality and homogeniety, asserts that $W_j(\cdot)$ is $\alpha$-concave for any $\alpha \in [-\infty,\frac{1}{n-j}]$. Therefore, Theorem~\ref{t:LpBorell-Brascamp-Lieb functionals} implies the following class of the $L_p$ Borell-Brascamp-Lieb inequalities for the $j$-th quermassintegrals of elements of $\mathcal{F}_{\gamma}(\R^n)$. 
	
	\begin{theorem}Let $p,q \in[1,\infty]$ be such that $1/p+1/q=1$, $t \in [0,1]$, and $j \in \{0,1,\dots,n-1\}$. Suppose that $\alpha \in [-1,\frac{1}{n-j}]$ and let $\gamma \in [-\alpha,\infty)$. Let $f,g \in \mathcal{F}_{\alpha}(\R^n)$.  Then we have
		\begin{equation*}\label{e:functionquermassBorell-Brascamp-Lieb 1}
		\begin{split}
		W_j((1-t) \times_{p,\alpha} f \oplus_{p,\alpha} t \times_{p,\alpha}g) \geq [(1-t) W_j(f)^{\be} + t W_j(g)^{\be} ]^{1/\be}, \quad \beta= \frac{p\alpha \gamma}{\alpha + \ga}.
		\end{split}
		\end{equation*}
		
	\end{theorem}
	
	\begin{definition}
		For any $f,g \in \mathcal{F}_s(\R^n)$, $j\in\{0, \ldots,n-1\}$, $s\in[-\infty, \infty],$ the $L_{p,s}$ mixed quermassintegral of $f, g\in\mathcal{F}_s(\R^n)$ with respect to the $L_{p,s}$ Asplund summations is defined as
		$$
		W_{p,j}^s(f,g):= \lim \limits_{\varepsilon \to 0}  \frac{W_j(f \star_{p,s}\varepsilon \cdot_{p,s} g) - W_j(f)}{\varepsilon},
		$$
		which is the first variation of the $j$-th \textit{quermassintegral} of function $f$.
	\end{definition}
	In particular, if $f=\chi_K$ for $K\in\cvxb_{(o)}$, $W_j(\chi_K)$ recovers the quermassintegral for convex bodies $K$, i.e., $W_j(K)$. Moreover, let $f=\chi_K$ and $g=\chi_L$ for $K,L\in\cvxb_{(o)}$, the $L_{p,s}$ mixed quermassintegral goes back to $p$-mixed quermassintegral for convex bodies in \cite{Lutwak1}.
	
	More generally, containing  the special cases of $s$-concave functions as special cases, we define for the generalized quermassintegral with functional $\Omega \colon \R_+ \to \R_+$ which is a bounded decreasing smooth function that decays faster than the exponential at infinity. Therefore, we further define the $\Omega$-$L_{p,s}$ mixed quermassintegral for base functions on $C_s(\R^n)$ as follows.   
	\begin{definition}{\emph{(General Quermassintegral for functions on $C_s(\R^n)$)}} \label{generalquermass}
		\quad
		\begin{enumerate}
			
			\item	The operator $I_{\Omega}\colon C_s(\R^n)  \to \R_+$  defined for $u\in C_s(\R^n)$ is the general $\Omega$-total mass 
			\[
			I_{\Omega}(u) := \int_{\R^n} \Omega(u(x)) dx. 
			\]
			\item For $j\in\{0,\dots,n-1\}$, the $\Omega$-$j$th-quermassintegral  is defined for $u \in C_s(\R^n) $ by 
			\[
			\mathbb{W}^{\Omega}_j(u) := c_{n,j} \int_{G_{n,n-j}} \int_H \Omega(u_H(x)) dx d\nu_{n,n-j}(H).
			\]
			\item The $\Omega$-$j$-th $L_p$-mixed quermassintegral of  $u,v \in C_s(\R^n) $ is defined as
			\[
				\mathbb{W}^{\Omega}_{p,j}(u,v) := \lim_{\e \to 0^+} \frac{\mathbb{W}^{\Omega}_j(u \boxplus_p(\e \boxtimes_p v)) - \mathbb{W}^{\Omega}_j(u)}{\e}.
			\]
			
		\end{enumerate}
	\end{definition}	
	
	Our next goal is an integral representation for $W^{\Omega}_{p,j}(f,g)$ for  functions $f=(1-su)_+^{1/s}, g=(1-sv)_+^{1/s}$ for $u,v\in C^{2,+}(\R^n)\subset C_s(\R^n)$ where
	\[
	C^{2,+}(\R^n) = \{u \in C_s(\R^n)  \colon \text{Hess }u(x) >0 \text{ for all } x\in \R^n\}.
	\]
	
	We need the following proposition which can be deduced from the Rockafeller's book \cite{Rock1} and \cite[Page 17]{Co2}.  
	
	\begin{proposition}\label{legendreproperty} Let $ u \in C^{2,+}(\R^n)$ and set $\varphi = u^*$.  Then the following hold true:
		\begin{enumerate}
			\item $\nabla u$ is a diffeomorphism;
			\item $\varphi \in C^2(\R^n)$;
			\item $(\nabla \varphi) = (\nabla u)^{-1}$;
			\item for every $y \in \R^n$,
			$
			\text{Hess }\varphi(y) = [\text{Hess }u(\nabla \varphi(y))]^{-1}
			$
			(here inverse is in the sense of matrices); in particular, $\text{Hess } \varphi(y) >0$ for all $y \in \R^n$;
			
			\item for every $y \in \R^n$
			\[
			\varphi(y) = \langle y, \nabla \varphi(y) \rangle - u(\nabla \varphi(y)).
			\]
			Analogously, for every $x \in \R^n$,
			\begin{equation}\label{legendreequation}
			u(x) = \langle x, \nabla u(x) \rangle - \varphi(\nabla u(x)). 
			\end{equation}
		\end{enumerate}
	\end{proposition}
	
	Let $p \geq 1$, $u \in C^{2,+}(\R^n)$, $\varphi = u^*$, and $\psi \in C_c^{\infty}(\R^n)$. For $\e > 0,$ we set $\varphi_{\e} =( \varphi^p+ \e \psi^p)^{1/p}$. There exists some $\bar{\e} >0$ such that $\varphi_\e \in C^{2,+}(\R^n)$ for all $\e \leq \bar{\e}$. For such $\e>0$, set $u_{\e} = (\varphi_{\e})^*$. We require the following lemma with respect to the variation formula for the projection function of $u_{\e}$.
	
	\begin{lemma}\label{variationbasep} Let $p\geq 1$, $u \in C^{2,+}(\R^n)$, $\varphi = u^*$, and $\psi \in C_c^{\infty}(\R^n)$, and fix $H \in G_{n,n-j}$ for $j\in\{0,\cdots,n-1\}$.  Set  $\varphi_{\e} =( \varphi^p+ \e \psi^p)^{1/p}$ for all $\e \leq \bar{\e}$, and $u_{\e} = (\varphi_{\e})^*$. Then, for every $x \in \text{int(dom}(u)|H)$, one has 
		\[
		\frac{d}{d \e}[(u_\e)_H(x)] = - \frac{d}{d\e}[(\varphi_{\e})_H(\nabla(u_\e)_H(x))].
		\]
		Moreover, for each $x \in \text{int(dom}(u)|H)$, one has 
		\[
		\left. \frac{d}{d \e}[(u_\e)_H(x)]  \right|_{\e=0}=- \frac{1}{p}\psi_H(\nabla u_H(x))^{p}\varphi_H(\nabla u_H(x))^{1-p}. 
		\]
	\end{lemma}
	
	\begin{proof} Fix $x \in \text{int(dom}(u)|H)$ and $\e >0$ sufficiently small.  Using (\ref{legendreequation}), we have
		\begin{align*}
		(u_\e)_H(x) = \langle x, \nabla (u_\e)_H(x) \rangle - (\varphi_\e)_H(\nabla (u_\e)_H(x)).
		\end{align*}
		Therefore, we obtain 
		\begin{align*}
		\frac{d}{d \e}[(u_\e)_H(x)]\Big|_{\e=0} &= \frac{d}{d\e}\left[ \langle x, \nabla (u_\e)_H(x) \rangle - (\varphi_\e)_H(\nabla (u_\e)_H(x))\right]\Big|_{\e=0}\\
		&= \left[\left\langle x, \frac{d}{d\e} \nabla (u_\e)_H(x) \right\rangle - \frac{d}{d\e}[(\varphi_\e)_H(\nabla (u_\e)_H(x))\right.\\
		&\left.-\left\langle \nabla (\varphi_\e)_H(\nabla (u_\e)_H(x)), \frac{d}{d\e} \nabla (u_\e)_H(x)\right\rangle\right] \Big|_{\e=0}\\
		&=- \frac{d}{d\e}[(\varphi_{\e})_H(\nabla(u_\e)_H(x))]\Big|_{\e=0}\\
		&=- \frac{1}{p}\psi_H(\nabla u_H(x))^{p}\varphi_H(\nabla u_H(x))^{1-p},
		\end{align*}
		where we have used the fact that $\nabla (u_\e)_H$ and $\nabla (\varphi_\e)_H$ are inverse of one another (Proposition \ref{legendreproperty} (3)). The second assertion follows form the fact that all functions involved are of class $C^{2,+}(H)$. 
	\end{proof}
	
	We require the following Blaschke-Petkantschin formula, which can be found in\cite{ShWeil}. 
	
	\begin{lemma}\label{t:BlashPetk}
		Let $H \in G_{n,n-j}$ for  $j\in\{1,\cdots,n-1\}$,  and $f \colon \R^n \to \R_+$ be a bounded Borel measurable function.  Then the following holds: 
		\[
		\int_{\R^n} f(x) dx = c_{n,j} \int_{G_{n,n-j}} \int_H f(x) \|x\|^{j} dx d\nu_{n,n-j}(H).\]
	\end{lemma}

	We are now prepared to establish the variational formula for the $\Omega$-$L_{p,s}$ mixed quermassintegral of functions on $C_{s}(\R^n)$ with the general quermassintegral in Definition \ref{generalquermass} based on the lemmas above.

	\begin{theorem} Let $j \in \{1,\dots,n-1\}$ and $H \in G_{n,n-j}$. Let $\Omega\colon \R_+ \to \R_+$ be a bounded smooth function such that  $\lim_{\|x\|\rightarrow\infty}\frac{\Omega'(x)}{\|x\|^{j}}=0$. Let $p \geq 1$, $j\in \{0,\dots, n-1\}$.  Then, for any $u \in C^{2,+}(\R^n) \cap C_c^{\infty}(\R^n)$ and $\psi \in C_c^{\infty}(\R^n)$, with $\varphi=u^*$ and $\psi = v^*$, the following holds: 
		\begin{equation}\label{integralforquermass}
			\mathbb{W}^{\Omega}_{p,j}(u,v) =- \frac{1}{p}\int_{\R^n} \frac{\Omega'(u(x))\psi_H(\nabla u_H(x))^{p}\varphi_H(\nabla u_H(x))^{1-p}}{\|x\|^{j}} dx.
		\end{equation}
	\end{theorem}
	\begin{proof}
		By definition of $\mathbb{W}^{\Omega}_{p,j}(u,v)$, we have 
		\begin{align*}
			\mathbb{W}^{\Omega}_{p,j}(u,v) &=\lim_{\e \to 0^+} \frac{	\mathbb{W}^{\Omega}_j(u \boxplus_p(\e \boxtimes_p v)) - 	\mathbb{W}^{\Omega}_j(u)}{\e}\\
		&= c_{n,j} \int_{G_{n,n-j}}\left(\lim_{\e \to 0^+}\int_H \frac{\Omega([u \boxplus_p ( \e \boxtimes_p v)]_H (x) )- \Omega(u_H(x))}{\e} dx \right) d\nu_{n,n-j}(H)\\
		&= c_{n,j} \int_{G_{n,n-j}}\left(\lim_{\e \to 0^+}\int_H \frac{\Omega([ u_H \boxplus_p(\e \boxtimes_p v_H)](x)) - \Omega(u_H(x))}{\e} dx \right) d\nu_{n,n-j}(H),
		\end{align*}
		where we have used the Proposition \ref{projectionforbase} and Lemma \ref{t:BlashPetk}. 
		
		For $\e>0$ sufficiently small, we see that $u_H \boxplus_p \e \boxtimes_p v_H
		\in C^{2,+}(\R^n) \cap C_c^{\infty}(\R^n)$, $\Omega(u_H)$ and $\Omega([ u_H \boxplus_p(\e \boxtimes_p v_H)])$ are integrable on $H$. Considering $B_r:=\{x \in H \colon \|x\| \leq r\}=B_r\cap H$, $r>0$, from the dominated convergence theorem, we see that 
		\begin{align*}
			\mathbb{W}^{\Omega}_{p,j}(u,v) &= c_{n,j} \int_{G_{n,n-j}}\left(\lim_{\e \to 0^+}\int_H \frac{\Omega([ u_H \boxplus_p(\e \boxtimes_p v_H)](x)) - \Omega(u_H(x))}{\e} dx \right) d\nu_{n,n-j}(H)\\
		&=c_{n,j} \int_{G_{n,n-j}}\left(\lim_{\e \to 0^+}\lim_{r \to \infty} \int_{B_r}\frac{\Omega([ u_H \boxplus_p(\e \boxtimes_p v_H)](x)) - \Omega(u_H(x))}{\e} dx \right) d\nu_{n,n-j}(H)\\
		&=c_{n,j} \int_{G_{n,n-j}}\lim_{r \to \infty} \int_{B_r}\left(\lim_{\e \to 0^+} \frac{\Omega([ u_H \boxplus_p(\e \boxtimes_p v_H)](x)) - \Omega(u_H(x))}{\e} dx \right) d\nu_{n,n-j}(H).
		\end{align*}
		By applying Lemma \ref{variationbasep}, we see that 
		\[
		\lim_{\e \to 0^+} \frac{\Omega([ u_H \boxplus_p(\e \boxtimes_p v_H)](x)) - \Omega(u_H(x))}{\e} = -\frac{1}{p}\Omega'(u_H(x))\psi_H(\nabla u_H(x))^{p}\varphi_H(\nabla u_H(x))^{1-p}. 
		\]
		Therefore,
		\begin{align*}
			\mathbb{W}^{\Omega}_{p,j}(u,v) &=   c_{n,j} \int_{G_{n,n-j}}\lim_{r \to \infty} \int_{B_r}\left(\lim_{\e \to 0^+} \frac{\Omega([ u_H \boxplus_p(\e \boxtimes_p v_H)](x)) - \Omega(u_H(x))}{\e} dx \right) d\nu_{n,n-j}(H)\\
		&= -\frac{1}{p} c_{n,j} \int_{G_{n,n-j}}\left(\lim_{r \to \infty} \int_{B_r}\Omega'(u_H(x))\psi_H(\nabla u_H(x))^{p}\varphi_H(\nabla u_H(x))^{1-p} dx \right) d\nu_{n,n-j}(H)\\
		&= -\frac{1}{p} c_{n,j} \int_{G_{n,n-j}} \int_H \Omega'(u_H(x))\psi_H(\nabla u_H(x))^{p}\varphi_H(\nabla u_H(x))^{1-p} dx d\nu_{n,n-j}(H)\\
		&=-\frac{1}{p} c_{n,j} \int_{G_{n,n-j}} \int_H \frac{\Omega'(u_H(x))\psi_H(\nabla u_H(x))^{p}\varphi_H(\nabla u_H(x))^{1-p}}{\|x\|^{j}}\|x\|^{j} dx d\nu_{n,n-j}(H)\\
		&= - \frac{1}{p}\int_{\R^n} \frac{\Omega'(u_H(x))\psi_H(\nabla u_H(x))^{p}\varphi_H(\nabla u_H(x))^{1-p}}{\|x\|^{j}} dx,
		\end{align*}
		where in the last step we have used Lemma~\ref{t:BlashPetk}. 
		
	\end{proof}
	
	\begin{remark} We remark that the right-hand side of identity \eqref{integralforquermass} may not be convergent. If we choose $\Omega$ such that $
		\lim_{\|x\| \to 0} \frac{\Omega'(u_H(x)) \varphi(\nabla u_H(x))^{1-p}}{\|x\|^{j}} < \infty$, (for example, in \cite[Theorem~5.7]{FXY}), when $\Omega(u)=e^{-u}$ and $j=0$,
		suppose that there exists a constant $k>0$ such that 
		\begin{equation} \label{compatible-1} \det\Big(\nabla^2 (u^*)^p (y) \Big)\leq k \big(u^*(y)\big)^{n(p-1)}  \det\big(\nabla^2 u^* (y) \big)
		\end{equation} holds for all $y\in \R^n\setminus\{o\}$, then the integral is  finite. 
	\end{remark}
	
	Here we list some special cases for formula (\ref{integralforquermass}) with typical parameters.
	Let $p\geq1,$ $j\in\{ 0,1,\dots,n-1\}$, $s \in (-\infty, \infty)$, and set $\Omega_s(r) = (1-sr)_+^{1/s}.$ Let $u,v\in C_s(\R^n)$. We denote
	\[
		\mathbb{W}_{p,j}^s(u,v) := \frac{p}{n-j}	\mathbb{W}^{\Omega_s}_{p,j}(u,v). 
	\]
	Consequently, we obtain the following corollary with respect to the $L_{p,s}$ mixed quermassintegral $W_{p,j}^s(f,g)$ based on the $\Omega_s$-$L_{p,s}$ mixed quermassintegral of $\mathbb{W}_{p,j}^{s}(u,v)$ above for base functions $u,v$ of $f,g$, respectively. That is,
	
	\begin{corollary}
For $p \geq 1$, $j\in \{0,\dots, n-1\}$, and $s \in \left(-\infty,\infty\right)$, let $f=(1-su)_+^{1/s}, g=(1-sv)_+^{1/s}$ such that $u,v\in C_s(\R^n)$ and $u \in C^{2,+}(\R^n)$, and $\psi \in C_c^{\infty}(\R^n)$ with $\psi = v^*$. Then the $L_{p,s}$ mixed quermassintegral for $f,g\in\mathcal{F}_s(\R^n)$ has the following integral representation: 
		\[
		W_{p,j}^s(f,g) =\frac{1}{n-j}\int_{\R^n} \frac{\left[1-su_H (x) \right]_+^{\frac{1}{s}-1} \psi_H(\nabla u_H(x))^p}{\|x\|^{j}}\varphi_H(\nabla u_H(x))^{1-p}dx
		\]
		For $s=0$, the above becomes
		\[
		W_{p,j}^0(f,g) =\frac{1}{n-j}\int_{\R^n} \frac{e^{-u_H(x)}\psi_H(\nabla u_H(x))^p \varphi_H(\nabla u_H(x))^{1-p}}{\|x\|^{j}}dx.
		\]
		
	\end{corollary}
	Furthermore, when $j=0$ and $p\geq 1$, it goes back to the results in \cite{FXY} by Fang, Xing and Ye where the formula (\ref{compatible-1}) holds. The author in \cite{Rotem2} also present an integral formula for $0<p<1.$ 
	If $\varphi=h_K(u)$ and $\psi=h_L(u)$ for $u\in S^{n-1},$ the support functions of two convex bodies $K,L\in\cvxb_{(o)}$, $j=0$ and $s=1$, it recovers the $L_p$ mixed volume for convex bodies $V_{p}(K,L)$ \cite{Lutwak1}, i.e.,
	\[
	V_p(K,L)=\frac{1}{n}\int_{S^{n-1}}h_{L}^p(u)h_K^{1-p}dS(K,u).
	\]
	\vspace{2mm}	
	
	\section{Acknowledgment}
	
The authors would like to thank Prof. Artem Zvavitch and Dr. Sergii Myroshnychenko for providing valuable suggestions and discussions during writing  of this paper.

	\vskip 2mm \noindent
	Michael Roysdon,   \ {\small \tt mroysdon@kent.edu} \\
	{\em 	School of Mathematical Sciences, Tel Aviv University,
		Israel}
	
	\vskip 2mm \noindent Sudan Xing, \ \ \ {\small \tt sxing@ualberta.ca}\\
	{\em Department of Mathematical and Statistical Sciences, University of Alberta,  Canada
			
	\end{document}